\documentclass[11pt]{article}
\usepackage{latexsym,amsfonts,amssymb,amsthm,amsmath,graphics,epsfig,extarrows}
\usepackage{arydshln,leftidx,mathtools}
\usepackage{comment} 
\allowdisplaybreaks
\usepackage{cancel}  
\usepackage[latin1]{inputenc}
\usepackage{times}
\usepackage[T1]{fontenc}
\usepackage{tikz}
\usepackage{mathrsfs}
\usepackage{ulem}
\usepackage[colorlinks=true, linkcolor=blue]{hyperref}

\usepackage{enumitem}  

\usetikzlibrary{arrows,shapes}
\usetikzlibrary{decorations.pathmorphing}
\usetikzlibrary{calc}

\usepackage[margin=1in]{geometry}

\usepackage[title,titletoc]{appendix}

\usepackage{empheq}
\usepackage{cases}

\usepackage{sectsty}
\sectionfont{\large}

\usepackage{xcolor}
\definecolor{CPG}{rgb}{0,0.6,0}

\newtheorem{theorem}{Theorem}[section]
\newtheorem{lemma}[theorem]{Lemma}
\newtheorem{proposition}[theorem]{Proposition}
\newtheorem{corollary}[theorem]{Corollary}
\newtheorem{definition}{Definition}[section]

\newtheorem{asp}{Assumption}

\theoremstyle{remark}
\newtheorem{remark}{Remark}[section]

\numberwithin{equation}{section}

\usepackage{everysel}

\EverySelectfont{%
	\fontdimen2\font=0.4em
	\fontdimen3\font=0.2em
	\fontdimen4\font=0.1em
	\fontdimen7\font=0.1em
	\hyphenchar\font=`\-
}
\usepackage{tgcursor}
\makeatletter
\g@addto@macro\normalsize{%
  \setlength\abovedisplayskip{5pt}
  \setlength\belowdisplayskip{5pt}
  \setlength\abovedisplayshortskip{5pt}
  \setlength\belowdisplayshortskip{5pt}
}
\makeatother

\usepackage{dsfont}
\usepackage{cases}
\usepackage[colorlinks=true]{hyperref}
\hypersetup{urlcolor=blue, citecolor=blue}

\newcommand{\R}{\mathbb{R}}
\newcommand{\N}{\mathbb{N}}
\newcommand{\cA}{\mathcal{A}}
\newcommand{\cB}{\mathcal{B}}
\newcommand{\cD}{\mathcal{D}}
\newcommand{\cE}{\mathcal{E}}

\newcommand{\cI}{\mathcal{I}}
\newcommand{\cH}{\mathcal{H}}
\newcommand{\cN}{\mathcal{N}}
\newcommand{\cM}{\mathcal{M}}

\newcommand{\mbf}[1]{\mathbf{#1}}
\newcommand{\dual}[2]{\langle#1,#2\rangle}
\newcommand{\abs}[1]{\left|#1\right|}
\newcommand{\lot}[2]{\underset{#1}{\mathrm{l.o.t.}}\{#2\}}

\title{
Attractors for  Second Order in Time  Non-Conservative Dynamics with Nonlinear Damping.}
\author{
	I.~Lasiecka\footnote{University of Memphis, Memphis, Tennessee} \\ email \href{mailto:lasiecka@memphis.edu}{lasiecka@memphis.edu} \and J.H.~Rodrigues\footnote{Okinawa Institute of Science and Technology, Japan} \\ email \href{mailto:jose-rodrigues@oist.jp}{jose-rodrigues@oist.jp}\and M.~Roy\footnote{North Carolina State University, Raleigh, North Carolina} \\ email \href{mailto:mroy5@ncsu.edu}{mroy5@ncsu.edu}
}
\date{\today
}

\theoremstyle{remark}

\usepackage{tgcursor}

\usepackage{dsfont}
\usepackage{cases}

\begin{document}

\maketitle

\begin{abstract}
A long-time behavior of solutions to a nonlinear plate model subject to non-conservative  and  non-dissipative effects and nonlinear damping is considered. The model under study is a prototype for a suspension bridge under the effects of unstable flow of gas. To counteract the unwanted oscillations a damping mechanism of a nonlinear nature is applied.  From the point of view of nonlinear PDEs, we are dealing with a non-dissipative and nonlinear   second order in time dynamical system of hyperbolic nature  subjected to nonlinear damping. One of the first goals is to establish {\it ultimate  dissipativity}  of all  solutions, which will imply  an  existence of a  {\it weak attractor}. The combined effects of non-dissipative forcing with nonlinear damping-leading to an overdamping-give rise to  major challenges in proving an existence of an absorbing set. Known methods based on equipartition of the energy do not suffice. A rather general novel methodology based on ``barrier's'' method  will be developed to address this and related problems. Ultimately, it will be shown that a weak attractor becomes strong, and the nonlinear PDE system has a coherent finite-dimensional asymptotic behavior. 
\end{abstract}
{\bf Keywords:}  Long time behavior, nonlinear plates dynamics, suspension bridges,  non-dissipative dynamical systems,  global and  finite dimensional  attractors. \hfill \break
	
\noindent{\bf 2010 MSC codes};  34B41, 34D05, 34D40, 34D45, 34H05, 35B35, 35B37, 35B40, 35B37, 74H40.  

\section{Introduction.}
{\bf  Physical Motivation}: One of the major problems in engineering and applied mathematics is the stability of civil structures such as bridges and buildings. 
Wind or flow of gas  can trigger sustained oscillations  which can lead  to a loss  of structural integrity. 
 This phenomenon is often referred to as ``flutter'' \cite{AGAR, dowell,dowell1}. The importance of this problem cannot be overstated, and examples are ubiquitous. The collapse of the Tacoma bridge in 1940 is just one prominent and eloquent example. The effects of torsional motion due to the fluttering effects caused by a strong wind led to a formidable collapse of the structure. The analysis of the models of suspension bridges and their real-life applications dates back to Timoshenko's work \cite{Tim43}. An analysis of the collapse of the Tacoma Narrows bridge in 1940,  \cite{GanderKwok} and references therein, brought up front a  realization that a reliable model for suspension bridges should be nonlinear with a  sufficient degree of freedom to exhibit torsional oscillations \cite{McKenna}. The first models introduced were based on one dimensional representation \cite{LM90,CM97,CJM91}, with the standard Dirichlet/hinged boundary conditions. However, these models fail to represent torsional oscillations in  the structures under consideration. 
References  \cite{Gazzola16,Gazzola15} were the first  ones to introduce the modeling of suspension bridges with hinged-free boundary conditions. These boundary conditions can be considered as the most relevant for suspension bridges as 
highlighted by B. Rayleigh in \textit{The Theory of Sound}: ``The problem of a rectangular plate, whose edges are free, is one of great difficulty, and has for the most part resisted attack.'' This statement captures the challenges associated with free boundary conditions, as further discussed in \cite{GanderKwok}. 

\ifdefined\xxxxxx
Hence, the problem of modeling and controlling the phenomenon has attracted a lot of attention  in the literature which was followed by the study of more accurate models. There are several different models \cite{DHMN03_9GAZ, FGIM24_GAZ10,GG_GAZ12,GAZ_14,PZ08_GAZ27} in the literature, depending on the physical phenomena involved. We also refer to \cite{BGLW22} for a recent survey of results with relevant mathematical analysis. 
\fi

Starting with \cite{Rocard}, and followed by a series of highly influential papers and books \cite{Gazzola015, Gazzola17, Gazzola16, Gazzola15, GAZ_14}, it has been agreed that a relevant model for a suspension bridge is governed by a nonlinear thin plate with hinged/free boundary conditions. This agrees with numerical and experimental studies which  provide an evidence for a correct description of vibrations resulting from external sources [such as a flow of gas, wind] \cite{DHMN03_9GAZ,GG_GAZ12,GAZ_14}.

\ifdefined\xxxx
One of the main forms of long-span bridge is a suspension bridge. Torsional forces exert strong pressure on the structure and cause fatigue of the material.
This eventually affects the dynamical stability, causing safety concerns. In order to accurately grasp the degradation and evolution law of the performance of a suspension bridge over its entire life cycle, it is necessary to accurately understand and describe its dynamical response. 
\fi

An essential point  is that the model under consideration should account for the effects of {\it unstable gas}. These are precisely the effects that are reinforced by the {\it non-linearity } of the system and the coupling between vertical and torsional displacements \cite{McKenna,Radu} which often leads to  a  degradation of the material. \hfill \break

 \noindent{\bf PDE model:} The model that motivates the present study is de facto a coupled model that involves a linearised  Euler equation around unstable equilibrium  and non-linear plate equations interacting at the interface \cite{dowell,dowell2,ChuLas10}. However, due to  Huygen's principle, it is possible to decouple the flow of gas from structural dynamics, which then results in a nonlinear elastic model influenced by gas flowing in a given  direction-say  $y$-causing structural stress. Such a model is derived in \cite{boutet} as valid for a sufficiently long time [related to the speed of propagation]. This motivates the present work where structural dynamics is perturbed by a  {\it non-conservative} and  {\it non-dissipative} feedback  force-see  $\beta \ne 0 $ in the model below. In fact, this term, seemingly very simple,  is the game changer in mathematical  analysis, rendering the system nonconservative, non-disssipative without a  gradient structure \cite{FerrGaz16}.  \\
Let $l>0$ and define $\Omega = (0,\pi)\times(-l,l)$.  The boundary of $\Omega$ consists of two parts defined on the same submanifold $\Gamma$ consisting of  $\Gamma_0$ and $\Gamma_1$  where $\Gamma_0 = [0,\pi] \times \{-l,l\} $ and $\Gamma_1 = \{0,\pi \}\times [-l,l] $. The norm in $L^2(\Omega)$ will be denoted by $\|\cdot\|_0$. For $T>0$, we consider the following extensible suspension--bridge type equation
\begin{equation}
    u_{tt} + \Delta^2 u +
    (\alpha - \delta\|u_x\|_0^2)u_{xx} + g(\|u_t\|_0)u_t 
    + f(u) + \beta u_y = 0,\quad\mbox{in}~Q^T:=(0,T)\times\Omega, \label{sbe}
\end{equation}
with {\it hinged}  boundary conditions on $\Gamma_1$ and {\it free}  on $\Gamma_0$. These  describe,  respectively, bending moments and shear forces -see \cite{BGLW22}. 
\begin{eqnarray}
    u =\Delta  u =0, &\mbox{on}~\Sigma_1^T:=(0,T)\times \Gamma_1,  \\
  \partial_{\nu} [ \Delta u  + (1-\sigma)  \partial_{\tau}^2 u ] = 0   &\mbox{on}~\Sigma_0^T:=(0,T)\times \Gamma_0 , \label{b2}
\end{eqnarray}
where $\nu$ (resp. $\tau$) represent normal (resp. tangential) direction to the boundary.
The initial conditions are given by 
\begin{equation}
    u(0,x) = u_0(x), \quad u_t(0,x) = u_1(x) \quad \mbox{in}~\Omega. \label{IC}
\end{equation}
In (\ref{sbe}),  $\alpha\! \in \!\R_+$ accounts for longitudinal prestressing [axial force], $\delta\!\geq\!0$ indicates the strength of restoring force resulting from stretching in the $x$-direction, $\beta\! \in \!\R$ is a flow parameter [in the $y$-direction] indicating the presence of   {\it non-conservative feedback  force}  affecting the plate, $f$ is a continuous function (in $\R$) representing an internal  restoring force  and $g(||u_t||_0)u_t  $  describes damping mechanism  which is {\it nonlinear, nonlocal} given by positive increasing  $ C^1(\R_+) $ function $g(s) $. 

The parameter $0<\sigma<1/2 $ in the boundary conditions \eqref{b2} denotes the Poisson ratio. In \eqref{sbe}, $u: Q^T \to \mathbb{R}$ represents the downward deflection of the midline [of unit length] in the vertical plane with respect to the reference configuration. Denoting $u^+(t,x) := \max\{0,u(t,x)\}$, the force $\kappa u^+$ ($\kappa$ constant) is a simplest function that models the restoring force of the stays in the suspension bridge \cite{MW87}. We define  $f(u) \equiv  \kappa u^+ +f_0 (u) $, where the source function [nonlinear]  $f_0\in C(R) $ can represent the bridge's weight per unit length or the external dynamic forces caused by passing vehicles or the wind. The nonlinear, non-local internal force describing elastic restoration with a prestrained  energy is given by ``Berger's'' type of nonlinearity  $\left(\alpha - \delta \|u_x\|^2\right) u_{xx}$,  \cite{Euler-Bernoulli}. \\

\noindent{\bf Comments about the model}:\\
1. A general model  (\ref{sbe})  with {\it $\beta =0$  and  with  Dirichlet or hinged boundary conditions } has been  investigated by several authors including  \cite{Aou21, MR4468373, Messaoudi22, LM90, Soh23,MZ05,YZ08,PK11,ZMS07}.  In fact, for  this {\it conservative and  dissipative} case there is an arsenal  of tools  in the area of dynamical systems \cite{MIRANVILLE} which can be used successfully. The non-conservative  case  $\beta \in R $,  with hinged/free boundary conditions and  a {\it linear} damping, $g(s) =constant$ has been studied only  recently in \cite{BGLW22}.  However, an important question: {\it how do  solutions of the model respond  to  unstability  caused by  the {\it  non-conservative} effects in  a presence of  {\it nonlinear}  over-damping, {\it has been unanswered.}} And this is the focus of the present paper
with dominant features   which distinguish this work from  the others, as follows: 
 { \textit{(i) free boundary conditions}}  on the long edge of the bridge, { \textit{(ii) the effect of a  non-conservative force  leading  to the loss of  dissipativity  and of a gradient structure}} along with  {\textit {(iii) a nonlinear, possibly degenerate effect  [$g(0) =0$]} of the damping}. These features, while motivated by the physics of the problem, provide for a  substantial new challenge from the point of view of the   mathematical analysis-as explained later. \\
2. The  feedback force  $\beta u_y $  is a  cause for the   loss of stability and promotes fluttering behavior. From the mathematical perspective, this term leads to the so called non-dissipative and non-conservative  dynamics for the resulting dynamical system. It is a major player  and a game changer. This term destabilizes the linearized dynamics. It should be noted that the model arises from the analysis of flow structure interaction where, over a long period, the effects of the 3-D flow  of gas translate into the term $(-u_t - U u_y )$, where $U$ is the normalized Mach number [determined by the local speed of the wind] \cite{boutet}. Thus,  the linear analysis would lead to a Hopf bifurcation at the onset of the fluttering speed. The nonlinear effects modulate oscillations which translate into sustained oscillations--potentially chaotic \cite{BGLW22, HTW18,dowell1,LasWebBal}. Thus, the presence of destabilizing effect $\beta u_y$ can not be underscored.

\ifdefined\xxxx
 However, realistic models exhibit nonlinearity of dissipation which is related to a ``stiffening'' of the plate.  It is also  known that {\it nonlinearity} of the damping often leads to the so called ``stability paradox''. The stronger the damping, does not translate into a higher level of stability. This phenomenon is referred to as  an ``over-damping'' which affects the equipartition of the energy. The issue is particularly sensitive in hyperbolic like dynamics where instability of an undamped system is of infinite-dimensional nature [in contrast to parabolic systems]. Various tools have been developed to handle over-damping. This includes ``epsilon'' Gronwall's inequality \cite{pata} -an elegant way of using Gronwall's inequality with a small parameter, or a contradiction type of arguments related to the so called ``barrier's method'' \cite{haraux,haraux1,haraux2,Memoir} used primarily in second order in time  hyperbolic dynamics. However, a combination of the ``over-damping'' with a lack of dissipativity and of gradient structure  makes the situation difficult. In view of this, the \textit{nonlinearity} of  dissipative mechanism and a  presence of {\it non-conservative force} provides for the analytical-mathematical challenge of the present paper. It should be also noted that  typical ways of handling non-dissipative/non-conservative  effects in the dynamics is either to assume suitable ``smallness''  of non-conservative term -\cite{Memoir} and references within- or to offset the undesirable effects by introducing into  a model a {\it strictly superlinear} force \cite{bociu}. Neither of these helping features appear in the present work where long time dynamics is studied  for a ``pure model''. Careful analysis of equipartition of energy and the effect of the over-damping on potential energy is the key player. This leads to a development of a new methodology, also inspired by \cite{haraux,haraux2},  within the area of dynamical systems.  The relevant result is formulated in an abstract form Theorem \ref{0.5.1}, so it could be applied within a much more general context of non-dissipative, non-conservative second order in time  PDE dynamics. 
\fi
In summary: the overriding goal of the paper is to investigate long time dynamics of the system generated by \eqref{sbe}--\eqref{IC}. This includes: well-posedness of solutions with bounded orbits, the existence of {\it weak and strong} attractors with  their properties. In particular, we shall show that in the case of non-degenerate damping at rest [$g(0) > 0$], the ultimate dynamics is finite-dimensional, with the possibility of chaotic behavior. This implies that the hyperbolic  nature of PDE dynamics is reduced asymptotically to a nonlinear ODE.

\section{Assumptions and Functional Settings.}\label{sec:Pre}
\subsection{Assumptions.}
The analysis of this paper relies on the following assumptions regarding the source and damping term.
\noindent\textbf{Assumption (f):}
Assumed that $f_0\in C^1$ verifies the following dissipativity type condition:
    \begin{equation}\label{aspf-1}
        \tilde{f}_0(s) = \int_0^s f_0(\tau)d\tau \geq - c s^2 - b,\quad\mbox{for some $c,b>0$.}
    \end{equation}
\begin{remark}
 A sufficient condition for \eqref{aspf-1} is the following  condition:
      $  \liminf_{|s|\to\infty}\frac{f_0(s)}{s} > -\infty.    $
\end{remark}

\noindent\textbf{Assumption (g):}
For the long time behavior, in order to control the orbits asymptotically in time, we shall resort to the following, rather simple assumption which, however, encompasses the main scenarios of interest such as superlinear damping [over-damping] at the infinity and the origin, as well as the damping with controlled behavior at the origin.
 \begin{equation}\label{aspg} 
 g(s) = b_0 + b_q s^q,~ b_0,~ b_q \geq 0 ~and ~b_0 + b_q \ne 0, q \geq 1.
 \end{equation}
\begin{remark}
 The assumption $g$ is a canonical form of nonlinear damping that can be easily generalized to other forms like $  m_0 s^q  \leq g(s)  - b_0 \leq M s^q, s>0$, or to the form $g(s)= \sum_{j=0}^q s^j , s>0 $-considered in the past \cite{Aou21}.  Most importantly, the overdamping case is present when
 $ b_q >0 $ and  damping controlled linearly from below  [$g(0) > 0$]  is manifested by the condition
   $b_0 > 0$. 
 This allows in-depth discussion of effects of the damping at the origin and infinity.
 \end{remark}
 \ifdefined\xxxxxxxx
\subsection{Definitions.} Having in mind the theory of dynamical systems governed  by evolution $S_t$  which is defined on a Hilbert space $\cH $ , the existence of global attractors, which describes the asymptotics of solutions to \eqref{sbe}-\eqref{IC}, can be established by means of two basic concepts, namely [1]. \textit{ultimate dissipativity} and [2]. \textit{asymptotic smoothness/compactness} of the evolution operator. For the reader's convenience, we recall the definitions below.
\begin{definition}[\bf Absorbing set]\label{def:Diss}
    A closed set $B\subset\cH$ is said to be \textit{absorbing} for $S_t$ if for any bounded set $B_0\subset\cH$, there exists $t_0=t_0(B_0)>0$ such that $S_t{B_0}\subset B$ for every $t\geq t_0$. Moreover, the evolution operator $S_t$ is said to be (bounded) \textit{dissipative} if it possesses a bounded absorbing set $B\subset\cH$.
\end{definition}
\begin{definition}[\bf Asymptotic smoothness]\label{def:AS}
	The evolution operator $S_t$ is said to be \textit{asymptotically smooth} if the following condition holds: for every bounded set $B_0\subset\cH$ such that $S_tB_0\subset B_0 $ for $t>0$ there exists a compact set $K\subset\overline{B_0}$ such that $S_tB_0\subset B_0$ converges uniformly to $K$ in the sense that
 \begin{equation*}
     \displaystyle\lim_{t\to\infty}\sup_{\mbf{y}\in B_0}dist_\cH(S_t\mbf{y},K)=0.
 \end{equation*}
\end{definition}
\begin{definition}[\bf Quasi-stability]\label{def:QS}
    The dynamical system $\{\cH,S_t\}$ is said to be \textbf{quasi-stable} on a set $B\subset\cH$, if there exists $t_\ast>0$, a Banach space $Z$, a globally Lipschitz mapping $K:B\to Z$ and a compact seminorm $n_Z$ on $Z$, such that for every $\mbf{y}_1,\mbf{y}_2\in B$ with $0\leq p<1$, we have
                \begin{equation*}
                           \|S_{t_\ast}\mbf{y}_1-S_{t_\ast}\mbf{y}_2\|_\cH
                          \leq p\|\mbf{y}_1-\mbf{y}_2\|_\cH + \textrm{n}_Z(K\mbf{y}_1-K\mbf{y}_2).
                \end{equation*}
\end{definition}
Quasi-stability is a powerful tool  used for characterization of   additional properties of attractors such as smoothness, dimensionality and  exponential attraction \cite{Memoir,LasMaMo} to be discussed later.
 \fi
 \subsection{Functional settings.}
\noindent We begin with the notation that will be used throughout. \\
$(\cdot, \cdot) \equiv ( \cdot, \cdot)_{L_2(\Omega)} $, $||u||_s \equiv ||u||_{H^s(\Omega) } $, where $H^s(\Omega) $ is a standard Sobolev space of order $s\in R $.

Following \cite{BGLW22}, we  introduce the  space:  $H^2_\ast = \left\{ u\in H^2(\Omega);~ u=0~\mbox{on}~\{0,\pi\}\times[-l,l] \right\}$, endowed with topology $\|u\|_{2,\ast}^2 = a(u,u)$, for $u\in H^2_\ast$ (equivalent to the usual Sobolev norm  $\|.\|_2$, where the scalar product $a(\cdot,\cdot)$ is given by
\begin{equation}
    \label{bifa} a(u,v) = \int_\Omega\left[ \Delta{u}\Delta{v} - (1-\sigma)\left(u_{xx}v_{yy} + u_{yy}v_{xx} - 2u_{xy}v_{xy} \right) \right] dx dy,\quad\mbox{for}~u,v\in H^2_\ast.
\end{equation}

Using this notation, we define a positive, self-adjoint operator $A: \cD(A) \subset L^2(\Omega) \to L^2(\Omega)$  equipped with {\it free/hinged}  boundary conditions: 
\begin{equation*}\begin{aligned}
    &A{u} = \Delta^2{u},  ~for~ u\in \cD(A) \mbox{ where}\\
    &\cD(A):= \!\left\{u \in  H^2_\ast, \Delta^2u\in L_2(\Omega) 
    \left|\begin{array}{cl} 
      \Delta   u = 0 & \mbox{on $ \Gamma_1 $ } \\
        u_{yy} + \sigma u_{xx}=0 = u_{yyy} + (2-\sigma)u_{xxy} & \mbox{on $ \Gamma_0 $}
    \end{array}\right.\!
    \right\}
\end{aligned}\end{equation*}
Due to specific configuration of the boundary conditions, one has $\cD(A) \subset H^4(\Omega)$, \cite{BGLW22,Gazzola15}.

The operator $A$ can be extended to a closed, positive self-adjoint operator as \newline $A:H^2_\ast\to(H^2_\ast)'$ by means of the following characterization
$a(u,v)=\dual{Au}{v},\mbox{ for}~u,v\in H^2_\ast.$\\
one also obtains $||u||_{\cD(A^{1/2} )}^2 = a(u,u) , u \in H^2_\ast $.\\
The nonlocal damping $g(||v||) v $  leads to a nonlinear operator  $D: \cD(A^{1/2}) \to L^2(\Omega)$, which is monotone and   expressed as $D{v}=g(\|v\|_0)v$ for every $v\in L^2(\Omega)$. Finally, the source operator is denoted by $F:H^2_\ast\to L^2(\Omega)$ with the following action
        \begin{equation}\label{F}
              F{u}=-\left[ (\alpha-\delta\|u_x\|_0^2)u_{xx} + f(u) + \beta u_y\right],\quad \mbox{for every } u\in H^2_\ast.
        \end{equation}
Note that $F$ has the form $F(u):=-\Pi^{\prime}(u)+N(u)$, with  non-conservative part $N(u) =-\beta u_y$ and conservative part  $\Pi(u)$, where 
 $\Pi'$  stands for Frechet derivative of the energy functional $\Pi:H^2_\ast\to\R$  given by:
        \begin{equation}\label{pi}
               \Pi(u) = \frac{\kappa}{2}\|u^+\|_0^2 - \frac{\alpha}{2}\|u_x\|_0^2+\frac{\delta}{4}\|u_x\|_0^4  + \int_\Omega \tilde{f_0}(u(x))dx \quad\mbox{for}~u\in H^2_\ast,
        \end{equation}
    With  the \textit{phase space} $\cH \equiv H^2_\ast\times L^2(\Omega)$ under  the norm 
$$\|\{u,v\}\|_\cH^2 = \|u\|_{2,\ast}^2 + \|v\|_0^2\mbox{ for }\{u,v\}\in\cH,$$ 
for each $\mbf{u}=\{u,v\}\in\cH$ the \textit{energy functional} $\cE:\cH\to\R$ associated with \eqref{sbe} is given by
        \begin{align}
              \cE(\mbf{u}) =\frac{1}{2}\left[\|u\|_{2,\ast}^2 + \|v\|_0^2 \right]+\Pi(u)= \frac{1}{2}\|\mbf{u}\|_{\cH}^2+\Pi(u). \label{nrg}
        \end{align}
 By {\it formal} calculations one obtains the following  {\it energy equality}, 
 \begin{align}
        \cE(\mbf{u}(t)) + \int_s^t g(\|u_t\|_0)\|u_t\|_0^2
        = \cE(\mbf{u}(s)) 
        - \beta\int_s^t\!\!\int_\Omega u_y u_t,\quad\mbox{for}~s\leq t. \label{nrg_id0}
    \end{align}
which displays  the presence of the non-conservative feedback force and the lack of dissipativity,

We shall show next that the system under consideration \eqref{sbe}-\eqref{IC} can be embedded into a class of more general 
abstract second order [in time]  {\it non-conservative and  non-dissipative } form.
With  $A$ described as before, $D: \cD(A^{1/2} ) \rightarrow \cD(A^{1/2})' $ monotone increasing, hemicontinuous  nonlinear operator,  $\Pi:\cD(A^{1/2} ) \rightarrow R $, a  $C^1$  functional and  nonconservative term $N : \cD(A^{1/2} )\rightarrow L_2(\Omega)$.
consider
        \begin{align}\label{sbe_abs} 
         \begin{cases}
               &u_{tt}(t) + A{u(t)} + D(u_t(t)) = F(u(t))=-\Pi^{\prime}(u)+N(u),\quad t>0, \\
               &\{u(0),u_t(0)\} = \{u_0,u_1\} \in \cH  \equiv  \cD(A^{1/2}) \times L_2(\Omega) 
         \end{cases}
        \end{align}    
        For the above abstract model one can formally write down  the energy relation in the form 
        \begin{equation}\label{Enf}
        \cE(\mbf{u}(t) ) + \int_s^t (D(u_t) , u_t) = \cE(\mbf{u}(s) ) + \int_s^t (N( u),u_t)
        \end{equation}
        with  the total [potenial and kinetic] energy :
        \begin{equation}\label{Enp} 
         \cE(\mbf{u}(t) ) = \frac{1}{2}  ||A^{1/2} u||_0^2  + \Pi(u) + \frac{1}{2}||u_t||_0^2.
         \end{equation}
    Thus,  long time behavior of the problem addressed  above amounts to showing that  the solutions of (\ref{sbe_abs} ) generate a dynamical system on  the phase space $\cH$  which admits  a global attractor with  the desired [eg smoothness, finite-dimensionality) properties.  This will allow to reduce asymptotically the abstract dynamics  (\ref{sbe_abs}) to an ODE. 
    \ifdefined\xxxxxxx
    With  the \textit{phase space} $\cH = H^2_\ast\times L^2(\Omega)$ equipped with the norm 
$$\|\{u,v\}\|_\cH^2 = \|u\|_{2,\ast}^2 + \|v\|_0^2\mbox{ for }\{u,v\}\in\cH,$$ 
for each $\mbf{u}=\{u,v\}\in\cH$ the \textit{energy functional} $\cE:\cH\to\R$ associated with \eqref{sbe_abs} is given by
        \begin{align}
              \cE(\mbf{u}) =\frac{1}{2}\left[\|u\|_{2,\ast}^2 + \|v\|_0^2 \right]+\Pi(u)= \frac{1}{2}\|\mbf{u}\|_{\cH}^2+\Pi(u). \label{nrg}
        \end{align}
The operator $\Pi$ is given by the expression \eqref{pi}. By formal calculations one obtains the following equality. 
 \begin{align}
        \cE(\mbf{u}(t)) + \int_s^t g(\|u_t\|_0)\|u_t\|_0^2
        = \cE(\mbf{u}(s)) 
        - \beta\int_s^t\!\!\int_\Omega u_y u_t,\quad\mbox{for}~s\leq t. \label{nrg_id0}
    \end{align}
which displays the lack of dissipativity,

\ifdefined\xxxxx
 can be expressed as $\Pi=\Pi_0+\Pi_1$; where
       \begin{align}
            &\Pi_0(u) = \frac{\kappa}{2}\|u^+\|_0^2 - \frac{\alpha}{2}\|u_x\|_0^2 + \frac{\delta}{4}\|u_x\|_0^4 + \int_\Omega\tilde{f}(u)+ c\|u\|_0^2+(b|\Omega|+\frac{\alpha^2}{4}), \label{pie0}\\
            &\Pi_1(u) = -c\|u\|_0^2-(b|\Omega|+\alpha^2/4)\quad \mbox{where $b,c$ are suitable constants.}
      \end{align}
Hence, the positive part of the energy $E(\mbf u)$ and the total energy $\cE(\mbf u)$ can be reformulated as,
       \begin{align}
               & E(\mbf{u})= \frac{1}{2}\|\mbf{u}\|_{\cH}^2+\Pi_0(u), \quad
               \quad \mbox{and} \quad \quad \cE(\mbf{u}) =E(\mbf{u})+\Pi_1(u). \notag \label{nrg}
       \end{align}
Later on we shall also consider linear part of the energy denoted by $$\tilde{E}(\mbf u)=\frac{1}{2}\|\mbf{u}\|_{\cH}^2=\frac{1}{2}\left[\|u\|_{2,\ast}^2 + \|v\|_0^2 \right].$$
\fi 
By formal calculations one obtains the following equality. 
 \begin{align}
        \cE(\mbf{u}(t)) + \int_s^t g(\|u_t\|_0)\|u_t\|_0^2
        = \cE(\mbf{u}(s)) 
        - \beta\int_s^t\!\!\int_\Omega u_y u_t,\quad\mbox{for}~s\leq t. \label{nrg_id0}
    \end{align}
which displays the lack of dissipativity, 
\fi   

\ifdefined\xxxxxxxx
 \subsection{Why is the problem interesting?}

   This becomes evident when we examine the abstract model presented in Section~\ref{sbe_abs}: \\
    {\bf 1. } A  very first step: {\it wellposedness } theory for finite energy solutions  depends on  the bounds obtained for the  energy function. In  our particular  case this amounts to showing that non-positive part of the {\bf  energy has a lower bound} ie:  $\Pi(u) $ is bounded from below when $||\mbf(u) ||_{\cH} \rightarrow \infty$. Such property typically depends on the ``strength'' of nonlinear term, hence depends on the more specific properties of the dynamics. We will show that in the case of the bridge model nonlinear elastic energy provides such mechanism-see Proposition \ref{P:1}. \\
  {\bf   2.} When considering a  long time behavior, {\bf boundedness in time} of solutions is a mandatory necessary ingredient. This property is typically deduced from energy relation. However, in the present case the nonconservative part of the dynamics  -term $F^* $-  precludes such conclusion. 
    In the case of the bridge model-this is the term with $\beta\in R $  responsible for instability of the system.
    This issue has been considered in the past literature by ``neutralizing'' the lack of dissipativity with ``large'' damping parameter \cite{Memoir}, or by assuming that the dissipation is bounded linearly \cite{BGLW22}, or by adding to potential energy strongly superlinear and coercive  term \cite{bociu}. 
    Neither of the above ``fixes'' will be considered in the present paper. \\
 {\bf  3.}  In order to assert existence of attracting sets,  the {\bf  ultimate bound for the orbits}  is necessary. And here, again, the lack of dissipativity becomes  the most serious obstacle. The system has no  {\bf gradient} structure.  To contend, we shall develop an abstract method which will resolve the issue of proving {\bf absorbing property}  for the dynamics without  the compromise of additional assumptions as specified in point 2-see  Theorem \ref{0.5.1}  and  Corollary \ref{ball} . \\
{\bf  4.} An existence of attractors requires convergence of the  orbits. This is connected to {\bf compactness}. 
    Due to the  hyperbolicity of dynamics, there is no natural smoothing mechanism - contrary to parabolic problems. On the other hand, the nonlinear terms are of critical nature [with respect to Sobolev's embeddings]. Thus,  any sort of compactness inherited from the  topology of nonlinear terms is out of question. 
    To contend, we shall use methods of compensated compactness with particular attention paid to 
    the energy functional, convexity of the nonlinear damping  and structure of the problem which would allow to unveil ``hidden'' convergence -see Theorem \ref{Asym-sm1}.  
    \fi
    
        \ifdefined\xxxxxx
Let us denote the \textit{phase space} $\cH = H^2_\ast\times L^2(\Omega)$ equipped with the norm 
$$\|\{u,v\}\|_\cH^2 = \|u\|_{2,\ast}^2 + \|v\|_0^2\mbox{ for }\{u,v\}\in\cH.$$ 
For each $\mbf{u}=\{u,v\}\in\cH$ the \textit{energy functional} $\cE:\cH\to\R$ associated with \eqref{sbe_abs} is given by
        \begin{align}
              \cE(\mbf{u}) =\frac{1}{2}\left[\|u\|_{2,\ast}^2 + \|v\|_0^2 \right]+\Pi(u)= \frac{1}{2}\|\mbf{u}\|_{\cH}^2+\Pi(u). \label{nrg}
        \end{align}
The operator $\Pi$, given by the expression \eqref{pi}, can be expressed as $\Pi=\Pi_0+\Pi_1$; where
       \begin{align}
            &\Pi_0(u) = \frac{\kappa}{2}\|u^+\|_0^2 - \frac{\alpha}{2}\|u_x\|_0^2 + \frac{\delta}{4}\|u_x\|_0^4 + \int_\Omega\tilde{f}(u)+ c\|u\|_0^2+(b|\Omega|+\frac{\alpha^2}{4}), \label{pie0}\\
            &\Pi_1(u) = -c\|u\|_0^2-(b|\Omega|+\alpha^2/4)\quad \mbox{where $b,c$ are suitable constants.}
      \end{align}
Hence the positive part of the energy $E(\mbf u)$ and the total energy $\cE(\mbf u)$ can be reformulated as,
       \begin{align}
               & E(\mbf{u})= \frac{1}{2}\|\mbf{u}\|_{\cH}^2+\Pi_0(u), \quad
               \quad \mbox{and} \quad \quad \cE(\mbf{u}) =E(\mbf{u})+\Pi_1(u). \notag \label{nrg}
       \end{align}
Later on we shall also consider linear part of the energy denoted by $$\tilde{E}(\mbf u)=\frac{1}{2}\|\mbf{u}\|_{\cH}^2=\frac{1}{2}\left[\|u\|_{2,\ast}^2 + \|v\|_0^2 \right].$$

By formal calculations one obtains the following equality. 
 \begin{align}
        \cE(\mbf{u}(t)) + \int_s^t g(\|u_t\|_0)\|u_t\|_0^2
        = \cE(\mbf{u}(s)) 
        - \beta\int_s^t\!\!\int_\Omega u_y u_t,\quad\mbox{for}~s\leq t. \label{nrg_id0}
    \end{align}
    \fi

\section{Main Results.}\label{mainresults}
The main results of this study focus on long time behavior of the dynamics, which can be characterized by attracting sets and their properties. The very first step toward this goal is to guarantee  global existence of {\it bounded in time }solutions for the associated dynamical system.
As we shall see, this will turn possible due to the nonlinearity of potential energy.
\subsection{Formulation of the results.}
\begin{definition}
    Let $T>0$. We say that $u:[0,T]\to L^2(\Omega)$ is a \textbf{strong solution} of \eqref{sbe}-\eqref{IC} on $[0,T)$ if $\mbf{u}=\{u,u_t\}\in C([0,T);\cH) \cap W^{1,1}_{loc}(0,T;\cH)$ with values $\mbf{u}(t)\in \cD(A)\times H^2_\ast$ for almost every $t \in [0,T]$, equation $\eqref{sbe_abs}$ is verified in $L^2(\Omega)$ for almost every $t\in[0,T]$ and the initial data in $\eqref{sbe_abs}$ hold.
    We say that $u$ is a \textbf{generalized solution} of \eqref{sbe}-\eqref{IC} on $[0,T]$ if $\mbf{u}=\{u,u_t\}\in C(0,T;\cH)$, the initial data in $\eqref{sbe_abs}$ hold and there exists $\{{u}^n\}_{n\in\N}$ sequence of strong solutions of \eqref{sbe}-\eqref{IC} on $[0,T)$ such that 
         \begin{equation}
                \lim_{n\to\infty}\max_{t\in[0,T]}\|\mbf{u}^n(t) - \mbf{u}(t)\|_\cH = 0.
        \end{equation}
\end{definition}

\begin{remark}
    Notice that, $\cD(A) \subset H^4(\Omega) $ \cite{BGLW22,Gazzola15}, thus if $u$ is a strong solution of \eqref{sbe}-\eqref{IC} on $[0,T)$ then $u(t)\in H^4(\Omega)\cap H^2_\ast$ for almost every $t\in[0,T]$.
\end{remark}

\noindent The wellposedness of generalised and strong solutions to \eqref{sbe}-\eqref{IC} is given below.

\begin{theorem}{\bf Wellposedness.}\label{thm:EU}
    Let $T>0$ and assume that $f_0$ satisfies \eqref{aspf-1} and $g(s) $ is monotone increasing, continuous on $\R_+$. Then, for each $\{u_0,u_1\}\in \cD(A)\times H^2_\ast$, there exists a unique \textbf{strong solution} $u$  of (\ref{sbe}) on the interval $[0,T]$ such that $\mbf{u}=\{u,u_t\}\in C(0,T;\cH)$ and $t\mapsto\mbf{u}_t(t)=\{u_t(t),u_{tt}(t)\}$ is weakly continuous in $\cH$. Moreover, 
    \begin{align}
        \mbf{u} \in L^\infty(0,T;\cD(A)\times H^2_\ast)\cap W^{1,\infty}(0,T;\cH); ~~ \cD(A) \subset H^4(\Omega)  \label{reg}
    \end{align}
    In addition, the following \textbf{energy identity}
    holds:
    \begin{align}
        \dfrac{d}{dt}\cE(\mbf{u}) + g(\|u_t\|_0)\|u_t\|_0^2
        = - \beta\int_\Omega u_y u_t,\quad\mbox{for}~t>0.
    \end{align}      
    If $\{u_0,u_1\}\in\cH$, then there exists a unique \textbf{generalized solution} $u$ to \eqref{sbe}-\eqref{IC} such that $\mbf{u}=\{u,u_t\}\in C(0,T;\cH)$ which  verifies the  integral form of the {\bf energy identity} 
    \begin{align}
        \cE(\mbf{u}(t)) + \int_s^t g(\|u_t\|_0)\|u_t\|_0^2
        = \cE(\mbf{u}(s)) 
        - \beta\int_s^t\!\!\int_\Omega u_y u_t,\quad\mbox{for}~s\leq t. \label{nrg_id}
    \end{align}
    In addition, each generalised solution is  a {\it weak} solution satisfying standard variational form of equation with test functions in $H^2_\ast$.
\end{theorem}

Theorem \ref{thm:EU} asserts  Hadamard wellposedness, hence an existence of the dynamical system $\{\cH,S_t\}$, where $S_t:\cH\to\cH$ and $t\geq0$, is given by $S_t\mbf{u}_0\equiv\mbf{u}(t)=\{u(t),u_t(t)\}$ the corresponding weak solution of \eqref{sbe}-\eqref{IC} with initial data $\mbf{u}_0\in\cH$. The proof is based on monotone operator theory \cite{Bar76,SHO97}, after  the bound from below of potential energy is asserted. The details are in Section \ref{wellposedness}. Of particular relevance is the fact that generalized solutions do satisfy energy {\it identity} and also weak variational form. \\

Long time behavior of the said dynamical system $(\cH,S_t) $ is discussed nest. 
\ifdefined\xxxxxx
\begin{definition}
 \hspace*{-.2cm}
    \setlist{nolistsep}
    \begin{itemize}
        \item A bounded closed subset $\mathfrak{A}\subset\cH$ is said to be a 
             \textbf{global attractor} for the evolution operator $S_t$ iff 
            \begin{itemize}
            \item[\textrm{i.}] $\mathfrak{A}$ is an invariant set, that is 
                 $S_t\mathfrak{A}=\mathfrak{A}$ for $t\geq0$;
            \item[\textrm{ii.}] $\mathfrak{A}$ is uniformly attracting; that 
                 is, for all bounded set $B\subset\cH$ 
                     \begin{equation*}
                            \lim_{t\to\infty}d_\cH\{S_tB|\mathfrak{A}\}=0,\quad \mbox{where $d_\cH$ is the Hausdorff semi-distance in $\cH$.}
                    \end{equation*}
             \end{itemize}
        \item Let $\mathfrak{A}$ be a compact global attractor for $\{\cH,S_t\}$. The 
              \textbf{fractal dimension} $\textrm{dim}_f\mathfrak{A}$ is defined by
              \setlength{\belowdisplayskip}{0pt} \setlength{\belowdisplayshortskip}{0pt}
              \setlength{\abovedisplayskip}{0pt} \setlength{\abovedisplayshortskip}{0pt}
              \begin{equation*}
                     \textrm{dim}_f\mathfrak{A} = \limsup_{\tau\to0}\frac{\ln n\{\mathfrak{A},\tau\}}{\ln(1/\tau)},
              \end{equation*}
              where $n\{\mathfrak{A},\tau\}$ is the minimal number of balls of radius $\tau$ which cover $\mathfrak{A}$.
        \item Let $\cN$ be the set of stationary solutions of \eqref{sbe}-\eqref{IC}, that is $\cN=\{\mbf{u}_0\in\cH;~S_t\mbf{u}_0=\mbf{u_0}\quad\mbox{for}~t\geq0\}$. We define the \textbf{unstable manifold} $\cM^u(\cN)$ emanating from $\cN$ as the set of all $\mbf{y}\in\cH$ such that there exists $\gamma=\{\mbf{u}(t):~t\in\R\}$ a full trajectory, that is $S_tu(s) = u(t+s)$ for every $t,s\geq0$, with properties $\mbf{u}(0)=\mbf{y}$ and $\lim_{t\to-\infty}dist_\cH(\mbf{u}(t),\cN)=0$.
    \end{itemize}
\end{definition}
\textcolor{red}{In order to control the orbits asymptotically in time, properties of the damping mechanism  $g(||u_t||_0) $ need to be more specific. We shall resort to the following, rather simple, assumption which however encompasses the main scenarios of interest  such as  superlinear damping [overdamping]  at the infinity and the origin, as well  the damping with controlled behavior at the origin.
\newline
\noindent\textbf{Assumption g.}
 \begin{equation}\label{aspg} 
 g(s) = b_0 + b_q s^q , b_0, b_q \geq 0 ~and ~b_0 + b_q \ne 0, q \geq 1.
 \end{equation}
 Assumption $g$  is a canonical form of nonlinear damping which can be easily generalised to other forms 
 like $  m_0 s^q  \leq g(s)  - b_0 \leq M s^q $, $ s >0$   , or to the form 
 $g(s) \sum_{j=0}^q s^j , s>0 $-considered in the  past  \cite{Aou21}.  Most importantly, the overdamping case is present when
 $ b_q >0 $ and  damping controlled linearly from below  [$g(0) > 0 $ ]  is manifested by the condition 
   $ b_0 > 0 $. 
 This allows in depth discussion of effects of the damping at the origin and infinity.}
\fi
For the evolution $S_t$, defined on a Hilbert space $\cH $ , the existence of global attractors, which describes the asymptotics of solutions to \eqref{sbe}-\eqref{IC}, can be established by means of two basic concepts, namely 1. \textit{ultimate dissipativity} and 2. \textit{asymptotic smoothness/compactness} of the evolution operator. For the reader's convenience, we recall the definitions below.
\begin{definition}[\bf Absorbing set]\label{def:Diss}
    A closed set $B\subset\cH$ is said to be \textit{absorbing} for $S_t$ if for any bounded set $B_0\subset\cH$, there exists $t_0=t_0(B_0)>0$ such that $S_t{B_0}\subset B$ for every $t\geq t_0$. Moreover, the evolution operator $S_t$ is said to be ultimately  \textit{dissipative} if it possesses a bounded absorbing set $B\subset\cH$.
\end{definition}
\begin{definition}[\bf Asymptotic smoothness]\label{def:AS}
	The evolution operator $S_t$ is said to be \textit{asymptotically smooth} if the following condition holds: for every bounded set $B_0\subset\cH$ such that $S_tB_0\subset B_0 $ for $t>0$ there exists a compact set $K\subset\overline{B_0}$ such that $S_tB_0\subset B_0$ converges uniformly to $K$ in the sense that
 \begin{equation*}
     \displaystyle\lim_{t\to\infty}\sup_{\mbf{y}\in B_0}dist_\cH(S_t\mbf{y},K)=0.
 \end{equation*}
\end{definition}
\begin{definition}[\bf Quasi-stability]\label{def:QS}
    The dynamical system $\{\cH,S_t\}$ is said to be \textbf{quasi-stable} on a set $B\subset\cH$, if there exists $t_\ast>0$, a Banach space $Z$, a globally Lipschitz mapping $K:B\to Z$ and a compact seminorm $n_Z$ on $Z$, such that for every $\mbf{y}_1,\mbf{y}_2\in B$ with $0\leq p<1$, we have
                \begin{equation*}
                           \|S_{t_\ast}\mbf{y}_1-S_{t_\ast}\mbf{y}_2\|_\cH
                          \leq p\|\mbf{y}_1-\mbf{y}_2\|_\cH + \textrm{n}_Z(K\mbf{y}_1-K\mbf{y}_2).
                \end{equation*}
\end{definition}
Quasi-stability is a powerful tool  used for characterization of   additional properties of attractors such as smoothness, dimensionality and  exponential attraction \cite{Memoir,LasMaMo} to be discussed later.

The {\it main} result of this paper is the following.
\begin{theorem}{\bf Global Attractors.} \label{thm:Att}
    Assume that $f_0$ satisfies \eqref{aspf-1}  and  $g$  satisfies assumption \eqref{aspg}. 
    Then, 
    {\bf Part I}: the dynamical system $\{\cH,S_t\}$ possesses a compact global attractor $\mathfrak{A}$. \\
   {\bf Part II}: Moreover, if $g(0)>0$  the system is quasi-stable. As a consequence, $\mathfrak{A}$ has finite fractal dimension $\mathrm{dim}_f\mathfrak{A}$ in $\cH$ and, in addition, $\mathfrak{A}$ is smooth in the sense that $\mathfrak{A}\subset H^4(\Omega)\cap H^2_\ast\times H^2_\ast\equiv Y$, and it is bounded in $Y$. Finally, if $\beta=0$, then $\mathfrak{A}$ coincides with unstable manifold $\mathcal{M}^u(\mathcal{N})$, with $\mathcal{N}$ denoting stationary set. 
\end{theorem}
\subsection{Discussion} 
{\bf About the problem studied.}
Realistic models exhibit nonlinearity of dissipation which is related to a ``stiffening'' of the plate.  It is also  known that {\it nonlinearity} of the damping often leads to the so called ``stability paradox''. The stronger the damping, does not translate into a higher level of stability. This phenomenon is referred to as  an ``overdamping'' which affects the equipartition of the energy.  On the other hand,  the  model is exposed to strong instability effects  expressed  by the  non-conservative term $\beta u_u$ which  is a ``proxy'' for the effects of an unstable flow [e.g., gas] moving over the plate with velocity $\beta$. As a result, the plate exhibits instabilities that cause flutter (see \cite{AGAR,dowell1}). In mathematical terms, loss of dissipativity which is a major obstacle in studying attractors within the framework of ``dissipative systems'' \cite{Gazzola015,MIRANVILLE,Mato,Mato1,yuming}.The issue is particularly sensitive in hyperbolic type of  dynamics where instability of an undamped system is of infinite-dimensional nature [in contrast to parabolic systems]. Various tools have been developed in the past  to handle overdamping. This includes ``epsilon'' Gronwall's inequality \cite{pata} -an elegant way of using Gronwall's inequality with a small parameter, or a contradiction type of arguments related to the so called ``barrier's method'' \cite{haraux,haraux1,haraux2,Memoir} used primarily in second order in time wave type  dynamics.  A popular way of treating  non-dissipative/non-conservative  effects in the dynamics is either to assume suitable ``smallness''  of non-conservative terms - \cite{Memoir} and references within - or to offset the undesirable effects by introducing into  a model a {\it strictly superlinear} force \cite{bociu}. Neither of these  features  giving  away the original problem appear in the present work. Long time dynamics is studied  for a ``pure model''- with no restrictions on the size of $\beta$  and no superlinearity requirement  imposed on the force $f$.  The model retains the basic features such as  nonlinear elastic energy  [$\delta >0$],   nonlinearity of the dissipation quantified  by $g(s)$ and instability due to external feedback force $\beta \ne 0 $.  \\

\noindent{\bf Why is the problem interesting?}
 From the perspective of  the abstract model presented in \eqref{sbe_abs}: \\
    {\bf 1. } A  very first step: {\it wellposedness } theory for finite energy solutions  depends on  the bounds obtained for the  energy function. In  our particular  case this amounts to showing that non-positive part of the {\bf  energy has a lower bound} ie:  $\Pi(u) $ is bounded from below when $||\mbf{u} ||_{\cH} \rightarrow \infty$. Such property typically depends on the ``strength'' of nonlinear term, hence depends on the more specific properties of the dynamics. We will show that in the case of the bridge model it is nonlinear elastic energy provides such mechanism-see Proposition \ref{P:1}. \\
  {\bf   2.} When considering a  long time behavior, {\bf boundedness in time of orbits} is a mandatory necessary ingredient. This property is typically deduced from energy relation. However, in the present case the nonconservative part of the dynamics  -term $N(u) =-\beta u_y  $-  precludes such conclusion. 
    This issue has been considered in the past literature by ``neutralizing'' the lack of dissipativity with  a ``large'' damping parameter \cite{Memoir}, or by assuming that the dissipation is bounded linearly \cite{BGLW22}, or by adding to potential energy strongly superlinear and coercive  term \cite{bociu}. 
    Neither of the above ``fixes'' will be considered in the present paper. \\
 {\bf  3.}  In order to assert existence of attracting sets,  the {\bf  ultimate dissipativity}  is necessary. And here, again, the lack of dissipativity along with genuine nonlinearity of the dissipation  become  the most serious obstacle. The system has no  {\bf gradient} structure.  To contend, we shall develop an abstract method which will resolve the issue of proving {\bf absorbing property}  for the dynamics without  the compromise of additional assumptions as specified in point 2-see  Theorem \ref{0.5.1}  and  Theorem \ref{ball} . \\
{\bf  4.} An existence of attractors requires convergence of the  orbits. This is connected to {\bf compactness}. 
    Due to the  hyperbolic nature of the dynamics, there is no natural smoothing mechanism - unlike  parabolic problems. On the other hand, the nonlinear terms are of critical nature [with respect to Sobolev's embeddings]. Thus,  any sort of compactness inherited from the  topology of nonlinear terms is out of question.
    To contend with, we shall use methods of compensated compactness with particular attention paid to 
    the energy functional, convexity of the nonlinear damping  and structure of the problem which would allow to unveil ``hidden'' convergence -see Theorem \ref{Asym-sm1}. \\
    {\bf 5.} The additional property  $g(0) > 0 $ allows to establish good asymptotic  estimate for the difference of two  solutions, which further upgrades properties of the attractor regarding its smoothness, finite-dimensionality as well as exponential  rate of attraction. \\
    
    \ifdefined\xxxxx
\noindent{\bf Literature.}
Rrealistic models exhibit nonlinearity of dissipation which is related to a ``stiffening'' of the plate.  It is also  known that {\it nonlinearity} of the damping often leads to the so called ``stability paradox''. The stronger the damping, does not translate into a higher level of stability. This phenomenon is referred to as  an ``overdamping'' which affects the equipartition of the energy. The issue is particularly sensitive in hyperbolic type of  dynamics where instability of an undamped system is of infinite-dimensional nature [in contrast to parabolic systems]. Various tools have been developed in the past  to handle overdamping. This includes ``epsilon'' Gronwall's inequality \cite{pata} -an elegant way of using Gronwall's inequality with a small parameter, or a contradiction type of arguments related to the so called ``barrier's method'' \cite{haraux,haraux1,haraux2,Memoir} used primarily in second order in time wave type  dynamics. It may be also noted that  typical ways of handling non-dissipative/non-conservative  effects in the dynamics is either to assume suitable ``smallness''  of non-conservative term - \cite{Memoir} and references within - or to offset the undesirable effects by introducing into  a model a {\it strictly superlinear} force \cite{bociu}. Neither of these helping features appear in the present work where long time dynamics is studied  for a ``pure model''.
 Careful analysis of equipartition of energy and the effect of the overdamping on potential energy is the key player. This leads to a development of a new methodology, also inspired by \cite{haraux,haraux2},  within the area of dynamical systems.  The relevant result is formulated in an abstract form Theorem \ref{0.5.1}, so it could be applied within a much more general context of non-dissipative, non-conservative second order in time  PDE dynamics. 
 In view of this, the \textit{nonlinearity} of  dissipative mechanism and a  presence of {\it non-conservative force}  $\beta u_y $ provides for the analytical-mathematical challenge of the present paper. There is also {\bf no restriction of smallness imposed on $\beta$}, the latter typical in a study of non-dissipative systems. As a consequence, the energy relation \eqref{nrg_id}  is not dissipative which is the major obstacle in studying attractors intrinsically based on the theory of ``dissipative'' systems, see \cite{Gazzola015,MIRANVILLE,Mato,Mato1,yuming}. From the modeling point of view, this term is a ``proxy'' for the effects of unstable flow [e.g., gas] moving over the plate with velocity $\beta$. As a result, the plate exhibits instabilities that cause flutter (see \cite{AGAR,dowell1}).  We will overcome this difficulty by developing dynamical system theory to account for {\it  non-dissipative effects in the presence of overdamping or underdamping}. One of the main difficulties is to show that the dynamical system in question admits a ``weakly compact'' attractor. Then, the ``upgrade'' of attractor to a global attractor is obtained by expanding on the methodology introduced in \cite{Memoir} and based on convexity methods-see also \cite{LasTat93}
    \fi  
    
\noindent{\bf Organization of the paper.}
After introducing preliminaries  in Section 4.1,  asserting the  wellposedness and generation of the dynamical system  $ (\cH, S_t) $ in Section \ref{wellposedness}, the first main ingredient of the proof of the Theorem \ref{thm:Att}  is  ultimate  boundedness in time of the orbits-Theorem  \ref{ball}. The challenge here is  due to the non-dissipative nature of the problem combined with the nonlinearity of the dissipation. These two features taken together  prevent application of the usual equipartition of kinetic and potential  energy. To contend with, we shall develop an abstract result  Theorem \ref{0.5.1}, based on barrier's method, which provides good condition for ultimate boundedness of the orbits.  And then, luckily enough, we will show that the said  abstract conditions are satisfied for the bridge model--culminating with a construction of ``weak attractor''. This will be done in subsection 5.1.1 culminating with--Theorem \ref{ball}.\\
 Then an ``upgrade'' of attractor to a global attractor-Theorem \ref{Asym-sm1}-provides for the content of subsection 5.1.2. This will be accomplished  through the proof of   asymptotic smoothness by  expanding on compensated compactness methods \cite{Memoir} along with  convexity approach 
inspired by \cite{LasTat93}.\\
Critical properties of the said attractor--Part II of Theorem \ref{thm:Att}  [including  regularity and finite dimensionality] are proved  in section \ref{attractor11}.  Specifically, subsections 5.2.2--Corollary \ref{cor:FFD} and  Corollary \ref{cor:Reg}. The latter follow from quasi-stability inequality obtained when  $g(0) > 0$ and proved in subsection 5.2.1.

%

\ifdefined\xxxxxx
\subsection{Prior literature on the topic.}
The analysis of the model of suspension bridges and their real-life applications dates back to Timoshenko's work \cite{Tim43}. When researchers examined the collapse of the Tacoma Narrows bridge in 1940,  \cite{GanderKwok} and references therein, they recognized that a reliable model for suspension bridges needed to be nonlinear with sufficient degrees of freedom to exhibit torsional oscillations \cite{McKenna, MIRANVILLE}. This approach enables a more accurate representation of real phenomena, which can be affected by various types of perturbations. Lazer and McKenna introduced the suspension bridge in the context of nonlinear analysis \cite{LM90}. The existence of weak solutions for nonlinear oscillation in a suspension bridge was shown in \cite{MW87}. Lazer and McKenna in the paper \cite{LM90} also studied the issue of one-dimensional nonlinear oscillation in suspension bridges. In \cite{CM97}, the authors studied the construction and stability of numerical solutions, and examined the interaction of colliding waves. Numerical results for the computation of periodic solution paths for a suspension bridge model were also presented in \cite{CJM91}, this work also shows the effect of the damping coefficient on the number of solutions. However, these one-dimensional models fail to represent torsional oscillations in suspension bridges. Numerical results discussing the nature of the collapse can be found in \cite{GP11}, also see \cite{CFN19,GWP18} for related results.

On the other hand, the analysis of the aerodynamic flutter effect on suspension bridges has been studied vastly in literature by several authors; see \cite{AGAR, Larsen, SCANLAN199651} and the references therein. The reader is referred to Chapter 2 in the book \cite{Fung} which briefly surveys flutter of non-streamlined structures such as suspension bridges. In a recent paper \cite{DIANA}, the aerodynamic instability was studied in a deck section model by specific tests performed in a wind tunnel and by numerical simulations. References \cite{Gazzola16,Gazzola15} were the first to introduce the modeling of suspension bridges with hinged-free boundary conditions. These boundary conditions can be considered as the most relevant for suspension bridges. In addition, we would like to highlight the significance of the free boundary problem, as Baron Rayleigh noted in \textit{The Theory of Sound}: ``The problem of a rectangular plate, whose edges are free, is one of great difficulty, and has for the most part resisted attack.'' This statement captures the challenges associated with free boundary conditions, as further discussed in \cite{GanderKwok}. As mentioned earlier, in  \cite{BGLW22} plate equation which takes into account  a) the impact of axial force on the vibration of hinged bars; see \cite{Euler-Bernoulli} 
b) effect of non-dissipative force caused by strong wind is considered. 
The long-time dynamics were fully analyzed in some of the above mentioned papers, and several stability results were derived in the presence of linear damping and a nonlinear source. There is a long literature on the topic that will be difficult to fully report. Focusing on more recent items, \cite{MZ05} established the existence of global attractors for the coupled system describing the vibration of a road bed in the vertical plane and that of the main cable, which together model the oscillations of a suspension bridge equation. We also refer to \cite{ZMS07} for global attractors for the suspension bridge equations; the results in this paper are obtained only for the case of linear damping. Recently, Yang and Zhong \cite{YZ08} studied the existence of a global attractor for the plate equation without assuming large values for the damping parameter. In \cite{PK11} the authors prove the existence of a global attractor for the suspension bridge equations for the dissipative model with nonlinear damping. However,  {\it none of these works takes into account nondissipative effects caused by the ``wind''-$\beta u_y$ term}. This term is the game changer. The theory of dissipative dynamics is no longer applicable. While the difficulty encountered may be somewhat offset by {\it linearity of the damping} \cite{BGLW22}, this is no longer possible in the case of {\it nonlinear $g(s)$}.  In fact, typical ways of dealing with non-dissipative dynamics is either to assume that  the non-dissipative term is suitably {\it small}  or that there is a  {\it strictly} superlinear force  driving the dynamics which has a correct ``sign'' and offsets non-dissipative effects for large frequencies.  This idea  was used in \cite{bociu}, where the existence of attracting sets has been proved for von Karman plates with boundary  dissipation which is  however assumed non-degenerate, linearly bounded at infinity   (see (1.10), (6.4), (1.8) in \cite{bociu}).Thus, the main challenge in the present work is to develop methodology,  which will lead to the resolution of the problem, which is {\it non-dissipative  without  any help from superlinear static damping, and with a strongly nonlinear  kinetic damping,  leading potentially to an overdamping and degeneration }. As it will be seen later, the issue appears already at the level of constructing ``weak'' global attractor. It suffices to notice that in the case of dissipative dynamics $\beta =0$, such a result can be obtained from the very general theory of gradient systems \cite{Chu15}. An additional challenge is presented by the fact that nonlinear damping may degenerate at the origin [$b_0 =0$ in the model]. In such a case, one still proves the existence of the global strong attractor by appealing to convexity theory \cite{LasTat93}.  In the non-degenerate case ($ b_0 > 0 $), such an attractor is both {\it smooth and finite dimensional} - as asserted by Theorem \ref{thm:Att}.

The reminder of this paper is devoted to the proofs of the main results.
    \fi
    
\section{Preliminaries.}
\subsection{Supporting Lemmas- bounds on potential energy.}
 We  begin by listing several results describing topological properties of the energy function. These will be used later in a critical manner for the proofs of the main results.
The  estimate  in Proposition \ref{P:0},  allowing to control ``low frequencies'' via the intrinsic  Berger's nonlinearity,  is fundamental  to the problem studied.

\begin{proposition}\label{P:0}
Let  $s\in(0,2]$. Then for every  $\eta>0$, there exists $C_{s,\eta}>0$ such that
        \begin{equation}\label{lem:LE-02}
              \|u\|_{2-s}^2
              \leq \eta\left[ ||u||^2_{H_{\ast}(\Omega) }+ \|u_x\|_0^4 \right] + C_{s,\eta},\quad\mbox{for every $u\in H^2_{\ast}(\Omega)$. }
         \end{equation}
\end{proposition}
\begin{proof}
This result has been used  and proved in \cite{BGLW22}. For completeness, we shall provide a more direct and complete argument.
\\
{\it Step 1}. We begin with the  Poincare type  of inequality. 
 Let $u\!=\!u(x, y, t)\! \in\! H^1(\Omega)$, where $\Omega\!=\![0, \pi] \times [-l, l]$  such that,
$u=0$ on $\{0, \pi \}\times [-l, l]$. Then 
\begin{equation}\label{x} 
 \int_\Omega |u|^2 \leq \pi^2  \int _\Omega |u_x|^2
 \end{equation}

Indeed, by the fundamental theorem of calculus, using the boundary conditions and applying Holder's inequality one obtains: 
  $  |u(x,y)|^2 \leq  \pi  \int_0^{\pi}  |u_x(z,y)|^2dz.  $
    Integrating both sides with respect to $y$ leads to 
    $ \int_{-l}^{l} |u(x,y)|^2 dy \leq {\pi}  \int_{-l}^{l}\int_0^\pi |u_x(z,y)|^2dz dy={\pi} \int_\Omega |u_x|^2 d\Omega. $
Integration in $x$ of both sides  of the  inequality above  gives 
    $\int_{\Omega} |u(x,y)|^2 dy dx \leq \left( {\pi}  \int_\Omega |u_x|^2 d\Omega \right)\int_0^\pi dx = \pi^2  \int_\Omega |u_x|^2 d\Omega $, hence \eqref{x}. 
 \\
 {\it Step 2}: 
We argue by contradiction.  Suppose that for some  $\eta  > 0$, there exists a sequence $u_n$ in $H^2(\Omega) $
such that,
\[\eta^{-1} \|u_n\|_{2-s }^2 -[a(u_n,u_n) +\|u_{x,n}\|_0^4 ] \rightarrow \infty.\] Then, we must have \(\|u_n\|^2_{2-s} \rightarrow \infty\). On the other hand, by rescaling $u_n$ as \(v_n:=\frac{u_n}{\|u_n\|_{2-s}}\) we obtain, \[\eta^{-1}  - [a(v_n,v_n) +||v_{x,n}||_0^4 ||u_n||^2_{2-s}  ] >0.\]
This implies: there is a constant $ M > 0 $ such that
\[\|v_n\|_0 \leq M ; \quad \|v_{x,n}\|_0^4 \|u_n\|^2_{2-s} \leq M, \mbox{ bounded uniformly in $n$}.\] Therefore, along a subsequence, we have 
\(v_n \rightharpoonup v ~weakly~in~H^2_\ast(\Omega).\) Additionally, due to the compactness of the embedding
\[H^2(\Omega)\! \subset \! H^{2-s}(\Omega), \quad s>0. \]
The convergence is actually strong in $H^{2-s}(\Omega)$. Since $v_n$ has been normalized, we conclude that the strong limit $v \ne 0$.
 On the other hand, $||u_n||^2_{2-s} \rightarrow \infty $ implies we must have 
  $ ||v_{x,n} ||_0 \rightarrow 0 $.
  This forces $ v_x \equiv 0 $ so $v= v(y) $. 
  However, since $v$ is  also in $H^2_{\ast}$, this implies by  (\ref{x})  $v\equiv 0 $, which contradicts $v \ne 0$, proving the Proposition.
\end{proof}

The proposition stated below asserts that the nonlinear energy of the system is controlled by the topology of the phase space,
as stated in \eqref{equivEn}.

The nonlinear operator \(F\), defined in \eqref{F}, takes the form \(F(u) = -\Pi'(u) + N (u) \). Additionally, we note  that $\Pi$
can be expressed as $\Pi=\Pi_0+\Pi_1$; where
       \begin{align}
            &\Pi_0(u) = \frac{\kappa}{2}\|u^+\|_0^2 - \frac{\alpha}{2}\|u_x\|_0^2 + \frac{\delta}{4}\|u_x\|_0^4 + \int_\Omega\tilde{f_0}(u)+ c\|u\|_0^2+(b|\Omega|+\frac{\alpha^2}{4}), \label{pie0}\\
            &\Pi_1(u) = -c\|u\|_0^2-(b|\Omega|+\alpha^2/4)\quad \mbox{where $b,c$ are suitable constants.}
      \end{align}
Hence the positive part of the energy $E(\mbf u)$ and the total energy $\cE(\mbf u)$ can be reformulated as,
       \begin{align}
               & E(\mbf{u})\equiv \frac{1}{2}\|\mbf{u}\|_{\cH}^2+\Pi_0(u), \quad
               \quad \mbox{and} \quad \quad \cE(\mbf{u}) =E(\mbf{u})+\Pi_1(u). \notag \label{nrg}
       \end{align}

 \begin{proposition}[Properties of the  functionals $\Pi_0, \Pi_1$] \label{P:1}  
 The following estimates hold:
  \begin{itemize}
            \item The functional  $\Pi_0 \geq 0$ and it is bounded on bounded sets of $H^2_{*}$.
            \item For every $\tilde{\eta} > 0, ~there~exists~a~ constant~ C_{\tilde{\eta}} \geq 0$ such that, 
            \begin{equation}\label{eq3}
                   |\Pi_1(u)| \leq \tilde{\eta}\left[\|u\|_{2,\ast}^2+\Pi_0(u)\right] + C_{\tilde{\eta}},\quad\mbox{for every $u \in \cD(A)^{1/2}$.}
            \end{equation}
        \end{itemize}
  \end{proposition}
  \begin{proof}We have
   $F =-\left[\kappa u^+ + (\alpha-\delta\|u_x\|_0^2)u_{xx} + f(u)\right]+N(u)$,  $N(u) = -\beta u_y$, and
            \begin{align}
                 \Pi(u)
                  &= \frac{\kappa}{2}\|u^+\|_0^2 - \frac{\alpha}{2}\|u_x\|_0^2 + \frac{\delta}{4}\|u_x\|_0^4 + \int_\Omega\tilde{f}(u) \notag\\
               & \geq  \frac{\kappa}{2}\|u^+\|_0^2 - \left(\frac{\epsilon}{2}\|u_x\|_0^4+\frac{\alpha^2}{4\epsilon} \right) + \frac{\delta}{4}\|u_x\|_0^4+ \int_\Omega\tilde{f}(u)\notag\\
              &\geq \frac{\kappa}{2}\|u^+\|_0^2+ \frac{\delta} 
              {8}\|u_x\|_0^4- c\|u\|_0^2 - \left(b|\Omega|+\frac{\alpha^2}{\delta}\right) \text{by \eqref{aspf-1} for $\epsilon=\frac{\delta}{4}.$} \label{eq1}
            \end{align}
    It follows from \eqref{eq1} that
    $\Pi_0(u)\geq \frac{\kappa}{2}\|u^+\|_0^2
        + \frac{\delta}{8}\|u_x\|_0^4\geq0$ for every $u\in H^2_\ast$ and
            \begin{align*}
                     \Pi_0(u) = \Pi(u) - \Pi_1(u) 
                      =& \frac{\kappa}{2}\|u^+\|_0^2 - \frac{\alpha}{2}\|u_x\|_0^2 + \frac{\delta}{8}\|u_x\|_0^4 + \int_\Omega\tilde{f}(u) + c\|u\|_0^2+(b|\Omega|+\frac{\alpha^2}{\delta}) \notag\\
                    &\leq C\left[1+\|u\|_{2,\ast}^2+\|u\|_{2,\ast}^4+\psi_f(\|u\|_{2,\ast})\right],
             \end{align*}
    where $C=\max\{\kappa/2 + {\alpha/2} +c , {\delta/8} , b|\Omega|+{\alpha^2/\delta}\}$ and $\psi_f:\R_+\to\R_+$ is a continuous increasing function given by $\psi_f(t) = |\Omega|\max_{|s|\leq t}|\tilde{f}(s)|$, $s\in\R_+$.
    Therefore, $\Pi_0$ is bounded on bounded sets of $H^2_\ast$. Finally, using Proposition \ref{P:0}, we have
            \begin{align}
                    \abs{\Pi_1(u)} 
                    &\leq c\|u\|_0^2 + b|\Omega|+\alpha^2/\delta \leq c\eta\left[a(u,u)+\|u_x\|_0^4\right] + cC_\eta + b|\Omega|+\alpha^2/\delta \notag\\
                   &\leq c\eta(1+8/\delta)\left[ a(u,u) + \Pi_0(u) \right] + cC_\eta + b|\Omega|+\alpha^2/\delta, \label{eq21}
            \end{align}
    for any $\eta>0$. 
    For $\eta=\tilde{\eta}/[c(1+8/\delta)]$ \vphantom{if $c\neq0$, in which way} we rewrite \eqref{eq21} as follows,
                 \begin{equation*}\label{eq3.}
                      |\Pi_1(u)| \leq \tilde{\eta}\left[\|u\|_{2,\ast}^2+\Pi_0(u)\right] + C_{\tilde{\eta}}\leq 2\cdot \tilde{\eta} E(\mbf u)+C_{\tilde{\eta}},\quad\mbox{for every $\tilde{\eta}>0$.}
                 \end{equation*}
   \end{proof}

\begin{lemma}[Bound from below]\label{lem:LE}
    Let assumption \eqref{aspf-1} be in place.
    There are constants $C_0>0$ and $C_1\in\R$ such that the energy functional $\cE$ defined in \eqref{nrg} satisfies the following estimate
           \begin{equation}\label{lem:LE_eq1}
                   C_0E(\mbf{u})-C_1 \leq \cE(\mbf{u}),\quad\mbox{for every}~\mbf{u}=\{u,v\}\in\cH.
            \end{equation}
\end{lemma}
\begin{proof}
    Let $\mbf{u}=\{u,v\}\in\cH$. It follows from expressions \eqref{nrg} and \eqref{pi}, and using Assumption \eqref{aspf-1} on $f$ along with the Proposition \ref{P:1} for $\tilde{\eta}=1/4$ we obtain:  
           \begin{align}
                \cE(\mbf{u})
                    = E(\mbf{u}) + \Pi_1(u) 
                     \geq(1/2)E(\mbf u)-C_1.\notag
             \end{align}
    This concludes the proof of the Lemma. 
  \end{proof}

\begin{lemma}[Bound from above]\label{lem:UE}
Let $f\in C(\R)$ and the assumption in $f$ \eqref{aspf-1} is in place. There are constants $C, M>0$ such that the following estimate holds
    \begin{equation}\label{lem:UE-1}
        \cE(\mbf{u}) \leq CE(\mbf u)+M.
    \end{equation}

\end{lemma}

\begin{proof}
    Let $\mbf{u}=\{u,v\}\in\cH$. Using the expression \eqref{pi} of $\Pi$ and proposition \ref{P:1} for $\tilde{\eta}=1/2$ and Young's inequality, we have
        $\cE(\mbf{u})
        = E(\mbf{u}) + \Pi_1(u) 
        \leq 2 \cdot E(\mbf{u})+C_{3}.$
      Hence the Lemma is proved.
\end{proof}

\begin{remark}
    Combining the Lemmas \ref{lem:LE} and \ref{lem:UE} we obtain, \begin{equation}\label{equivEn}
            \frac{1}{2}E(\mbf u)-C_1 \leq \cE(\mbf u) \leq 2 \cdot E(\mbf u)+ C_2.       
    \end{equation}
\end{remark}

\subsection{Well-posedness of locally  bounded solutions - Proof of Theorem \ref{thm:EU}}\label{wellposedness}
In this section, we briefly present the proof of Theorem \ref{thm:EU},  which is based on nonlinear semigroup theory \cite{Bar76,ChuLas10,SHO97}. In the case of linear damping, the wellposedness is known from \cite{BGLW22}. The latter leads to the constructions of the dynamical system on the phase space $\cH$. 

With reference to the abstract form \eqref{sbe_abs}, we first prove the following properties regarding the operators entering the formulation in (\ref{sbe_abs}).

\begin{lemma}\label{lem:op_D}
The damping operator $D:H^2_\ast\to(H^2_\ast)'$ is monotone hemicontinuous.
\end{lemma}
\begin{proof}
    Let $u,v\in H^2_\ast$ be given. By monotonicity of $g(|s|) s $ one concludes.
                \begin{equation*}\begin{aligned}
                     \dual{Du\!-\!Dv}{u\!-\!v} 
                         &\!= \!(g(\|u\|_0)u \!-\! g(\|v\|_0)v,u\!-\!v)_\Omega\! \geq 0. 
    \end{aligned}\end{equation*}
    which shows that $D$ is monotone. Regarding the hemicontinuity, it follows from continuity of  $g(s)$ on $\R_+$ 
    that $g$ is continuous on $L_2(\Omega)$   and therefore the map is also continuous$$\lambda\mapsto\dual{D(u+\lambda v)}{v}=g(\|u+\lambda v\|_0)[(u,v)_\Omega+\lambda\|v\|_0^2],$$ 
    where $u,v\in H^2_\ast$ are arbitrarily taken. This shows that $D$ is hemicontinuous.
\end{proof}

\begin{lemma}\label{lem:op_F}
    Assume that $f_0\in C^1(\R)$. Then, $F:H^2_\ast\to L^2(\Omega)$ is locally Lipschitz.
\end{lemma}
\begin{proof}
    Let $R>0$ and $u,v\in H^2_\ast$ such that $\|u\|_{2,\ast},\|v\|_{2,\ast}\leq R$. Hence, using the representation of   $F$
    we have
    \begin{align}
        \|Fu-Fv\|_0
        &\leq \kappa\|u^+-v^+\|_0
        + \alpha\|u_{xx}-v_{xx}\|_0
        + \delta\|\|u_x\|_0^2u_{xx} - \|v_x\|_0^2v_{xx}\|_0
        + \|f_0(u)-f_0(v)\|_0 \notag\\
        &+ |\beta|\|u_y-v_y\|_0 \notag\\
        &\leq
        \kappa \|u-v\|_0
        + \alpha\|u_{xx}-v_{xx}\|_0
        + \delta\|u_x\|_0^2\|u_{xx}-v_{xx}\|_0
        + \delta\|v_{xx}\|_0\abs{\|u_x\|_0^2 - \|v_x\|_0^2} \notag\\
       & + \left\{\max_{|s|\leq2R}|f_0'(s)|^2\right\}^{1/2}\|u-v\|_0
        + |\beta|\|u_y-v_y\|_0 \notag\\
        &\hspace*{3pt} \leq \underbrace{\left[\kappa+|\alpha|+\delta R(1+2R)+ \left\{\max_{|s|\leq2R}|f'_0(s)|^2\right\}^{1/2} + |\beta|\right]}_{L_{f,R}}\|u-v\|_{2,\ast}
    \end{align}
    and consequently $\|Fu-Fv\|_0\leq L_{f,R}\|u-v\|_{2,\ast}$, for every $u,v\in H^2_\ast$. This shows that $F$ is locally Lipschitz from $H^2_\ast$ to $L^2(\Omega)$. 
\end{proof}

We are now in position to conclude the proof of the  well-posedness for system \eqref{sbe}-\eqref{IC}.

\begin{proof}[Proof of Theorem \ref{thm:EU}:]
    Considering the abstract form \eqref{sbe_abs}, it is known that $A$ is closed, linear positive and self-adjoint acting on $H^2_\ast$. Lemmas \ref{lem:op_D} and \ref{lem:op_F} above ensures that $D$ is monotone hemicontinuous while $F$ is locally Lipschitz from $H^2_\ast$ into $(H^2_\ast)'$. Furthermore,  $N$ is globally Lipschitz from $H^2_\ast$ into $(H^2_\ast)'$ (see Lemma \ref{lem:op_F}). The Lemmas \ref{lem:LE} and \ref{lem:UE} together with the above mentioned properties of the operators $A, D, F \mbox{ and } N $ verify the assumption 2.4.15 in \cite{ChuLas10} page 77, therefore, applying Theorem 2.4.16 page 78, we conclude that for every $T>0$ and every $\mbf{u}_0=\{u_0,u_1\}\in \cD(A)\times H^2_\ast$, there exists a unique strong solution $\mbf{u}=\{u,u_t\}$ in the class \eqref{reg} and satisfying the energy identity \eqref{nrg_id}. The regularity $D(A) \in H^4(\Omega) $ follows from \cite{Gazzola15}.\\
Furthermore, since $D$ is a bounded operator on bounded sets of $L^2(\Omega)$ (see Proposition 2.4.21 in \cite{ChuLas10} page 83), we conclude that if $\mbf{u}_0\in\cH$ then there exists a unique generalized solution $\mbf{u}$ satisfying the energy identity \eqref{nrg_id}. This implies Hadamard wellposedness [continuous dependence on the data] of all generalized solutions. Moreover, each generalized solution satisfies standard variational form with test functions in $H^2_\ast$.  \\
\end{proof}

\section{Attractor--Proof of Theorem \ref{thm:Att}.}\label{prfAtt}

The proof of Theorem \ref{thm:Att} requires several steps and will be divided into the following subsections.
\ifdefined\xxxxxxx
\textcolor{red}{Having in mind the theory of dynamical systems, the existence of global attractors, which describes the asymptotics of solutions to \eqref{sbe}-\eqref{IC}, can be established by means of two basic concepts, namely [1]. \textit{ultimate dissipativity} and [2]. \textit{asymptotic smoothness/compactness} of the evolution operator. For the reader's convenience, we recall the definitions below.
\begin{definition}\label{def:Diss}
    A closed set $B\subset\cH$ is said to be \textit{absorbing} for $S_t$ if for any bounded set $B_0\subset\cH$, there exists $t_0=t_0(B_0)>0$ such that $S_t{B_0}\subset B$ for every $t\geq t_0$. Moreover, the evolution operator $S_t$ is said to be (bounded) \textit{dissipative} if it possesses a bounded absorbing set $B\subset\cH$.
\end{definition}
\begin{definition}[\bf Asymptotic smoothness]\label{def:AS}
	The evolution operator $S_t$ is said to be \textit{asymptotically smooth} if the following condition holds: for every bounded set $B_0\subset\cH$ such that $S_tB_0\subset B_0 $ for $t>0$ there exists a compact set $K\subset\overline{B_0}$ such that $S_tB_0\subset B_0$ converges uniformly to $K$ in the sense that
 \begin{equation*}
     \displaystyle\lim_{t\to\infty}\sup_{\mbf{y}\in B_0}dist_\cH(S_t\mbf{y},K)=0.
 \end{equation*}
\end{definition}}
\fi 
The first part of Theorem \ref{thm:Att}--Part I--will be obtained by using Theorem 2.3.5 on \cite{Chu15} page 63,  as soon as   one shows that $\{\cH,S_t\}$ is an {\it ultimately dissipative} and  {\it asymptotically smooth} dynamical system. It is here where non-conservative effects  along with an overdamping are the game changers. The second part of Theorem \ref{thm:Att}--Part II--will be proved by establishing quasi-stability property. 

\begin{remark}
    In fact, Theorem 2.3.5 in \cite{Chu15} is stated for asymptotic compact dynamical systems (see Definition 2.2.1 in \cite{Chu15} page 52). However, the equivalence between asymptotic smoothness and compactness is established in Proposition 2.2.4 in the same reference.
\end{remark}
\subsection{Proof of Part I of Theorem \ref{thm:Att}} 

As a very first step toward  the proof of the existence of attractors, we will establish an  existence of a ``weakly compact'' attractor. For this, we need to assert the existence of a positively invariant bounded  {\it absorbing}  set for the evolution operator $S_t$. Due to the ``lack of dissipativity'', evidenced by the presence of longitudinal wind flow [case $\beta\ne 0$], the system is {\it not of the gradient type.} This forces one to construct a suitable absorbing ball. At this point, the main challenge occurs. It suffices to note that standard approaches based on the construction of a suitable Lyapunov function are not adequate. This is mainly due to the presence of a {\it non-conservative term} along with the presence of a {\it non-linear damping}. In fact, the nonconservative term alone has been discussed recently in \cite{BGLW22}  by constructing a suitable Lyapunov function. However, this construction fails when dealing simultaneously with non-linear damping. New methodology must be developed. We shall first proceed at the abstract level in order to single out the difficulty. We will develop an abstract result subject to a number of well-calibrated technical assumptions. In the final step, we shall show that these assumptions are satisfied by the bridge model. At the same time, the developed methodology may be applicable to other problems that are genuinely non-dissipative with a nonlinear dissipation-a large class of applied models of general interest.

\subsubsection{Ultimate Dissipativity -- Existence of weak attractor}

 The construction of the absorbing set will be based on the following more general result given in Theorem \ref{0.5.1}, which specifically addresses the situation of {\it  non-conservative  and non-dissipative  dynamics with nonlinear damping}.

 \paragraph{Ultimate Dissipativity for  a model with non-conservative force -- Abstract framework:} 
 
Consider the following model  -- as in (\ref{sbe_abs} ) under Assumption \ref{asp1} and  Assumption \ref{asp2} stated below:
\begin{align}\label{model2}
    \begin{cases}
    u_{tt}(t)+\cA u(t)+kD(u_t(t))=F(u(t)),\\
    u|_{t=0}=u_0 \in \cD(\cA^{1/2}), u_t|_{t=0}=u_1 \in L_2(\Omega) 
    \end{cases}
\end{align}
where $\cA$ is closed, positive, self-adjoint operator acting on  $L_2(\Omega) $, with $\cD(\cA) \subset L_2(\Omega) $. 
We will use a standard notation: $||u||\equiv |u|_{L_2(\Omega)} $, $(u,v)\equiv (u,v)_{L_2(\Omega)}$ for any $u, v\in L_2(\Omega)$. 

\begin{asp}\label{asp1}
\begin{itemize}[noitemsep]
    \item [i)] The operator 
    $D: \cD(\cA^{1/2} ) \rightarrow [\cD(\cA^{1/2} )]' $ is assumed monotone and hemicontinuous with
    $D(0) =0, $
    
    \item  [ii)] the nonlinear operator
    $F:\cD(\cA^{1/2}) \to  L_2(\Omega)  $ is locally Lipschitz and  has the form, $F(u)=-\Pi^\prime(u)+ N(u)$  where $\Pi(u)=\Pi_0(u)+\Pi_1(u)$, is a $C^1$ functional on $\cD(\cA^{1/2})$, $\Pi^\prime$ stands for the Frechet Derivative of the functional $\Pi$. We also assume that,
    \begin{itemize}[noitemsep]
        \item [a)] $\Pi_0(u) \geq 0$ is locally bounded on $\cD(\cA^{1/2})$,
        \item [b)] for every $\eta >0$ there exists a constant $c_\eta \geq 0$ such that,
                \begin{equation}\label{1.5}
                     |\Pi_1(u)| \leq \eta \cdot \left(|\cA^{1/2}u|^2+\Pi_0(u) \right)+c_\eta \text{ for $u \in \cD(\cA^{1/2}),$}       
                \end{equation}
    \end{itemize}
    \item [iii)] the mapping $N(u)$  satisfies the following bound:
                \begin{equation}\label{A1-F.AST}
                         (N(u), u_t) \leq c_1+c_2\big[ ||\mbf(u) ||^2_{\cH } +  \Pi_0(u)\big] + k_0(Du_t,u_t),
                \end{equation} 
       for any $u, u_t \in \cD(\cA^{1/2})$, where $c_1, c_2 > 0$ and $0 \leq k_0 < k$ are constants.
\end{itemize}
\end{asp}
The dynamical system $(\cH,S_t)$ governed by the above model under Assumption \ref{asp1} is wellposed  on $\cH= \cD(\cA^{1/2} ) \times L_2(\Omega) $ and satisfies the following energy relation,
\begin{equation*}
 \cE(u(t),u_t(t))+k\int_0^t(Du_t(s),u_t(s))ds= \cE(u_0,u_1)+\int_0^t(N(u(s)),u_t(s))ds
\end{equation*} 
\begin{equation*}
\text{where, }
\begin{cases}
     E(u(t),u_t(t))=\frac{1}{2}\left[||\cA^{1/2}(u)||^2+ ||u_t||^2\right]+\Pi_0(u), \\
    \cE(u(t),u_t(t))=E(u(t),u_t(t))+\Pi_1(u). 
\end{cases} 
\end{equation*}
In addition, we impose the following  hypotheses:
\begin{asp}\label{asp2}
\begin{itemize}[noitemsep]
    \item [i)] There exist constants $c_0 \geq 0$ and $c_1 >0$ such that with $0\leq \gamma <1.$
    \begin{equation}\label{A1}
    ||u_t||^2  \leq c_0 +c_1(Du_t,u_t) \text{ for any } u_t\in \cD(\cA^{1/2}),  
    \end{equation}
   \item [ii)] there exists $0 \leq \eta <1$, constants $c_2>0$ 
               and $c_3, c_4 \geq 0$ such that for any $u, u_t \in \cD(\cA^{1/2})$
                     \begin{align}\label{A2}
                            -(D u_t,u)+( F^\ast(u), u) \!\leq\!  (\Pi ^{\prime}(u),u)+\eta|\cA^{1/2}u|^2-c_2\Pi_0(u)+c_3 +c_4.\left[ 1+E(u,u_t)\right]^{\gamma}(D u_t,u_t),
                      \end{align}   
     \item [iii)] we assume the relation
                     \begin{equation}\label{A3b}
                            ( N(u), u_t) \leq \eta \kappa(Du_t,u_t)+\delta E(u,u_t)+b(\frac{1}{\delta}),
                     \end{equation}
        for some number $0\leq \eta <1$ and a non decreasing continuous function  $b: \mathbb{R}_+\to \mathbb{R}_+$ such that it satisfies the following {balancing condition }
          \begin{equation}\label{A3a}
                 \lim_{x\to \infty}\Big\{x^{1-\frac{1}{\gamma}}b(x)\Big\}=0,
          \end{equation}
         in the case of $\gamma>0$ where $\gamma$ is the parameter from assumption above and for any $0< \delta \leq 1$. 
\end{itemize}
\end{asp}

\begin{remark}\label{remark5}
In the case where the damping is linear $D(u_t) = k u_t$, one can take $\gamma = 0$. In this case the balancing condition \eqref{A3a} is automatically satisfied.

Note that conditions (i) and (iii) refer to the estimates of kinetic energy. Condition (ii) instead, refers to the equipartition of potential and kinetic energy. The key player is the parameter $\gamma < 1 $  along with the ``balancing condition''  in  \eqref{A3a},the latter  quantifies the interplay between the loss of kinetic energy due to nondissipativity and the loss of potential energy due to an overdamping. The presence of genuinely non-linear damping introduces potential energy on the right-hand side of \eqref{A2} with a constant power $\gamma$. A priori, this quantity is unbounded. On the other hand, the nondissipative effect of $N(u)$ on the right side of the inequality \eqref{A3b} leads to the term ``b(x)'', which may be growing if $\delta$ is small - a quantity needed in absorbing potential energy. This leads to the competing effects caused by the lack of dissipativity of the term $N(u)$ and non-linear damping as exhibited in \eqref{A2} and \eqref{A3b}. The balancing condition \eqref{A3a} is critical in achieving the equilibrium. Also note that the case where $\gamma=0$, which is for linear damping, \eqref{A3a} is trivially true, and in the absence of a non-dissipative term, one can take $b(s)\! =\!0$. So, the classical cases of the analysis in dissipative dynamical systems or the non-dissipative dynamical systems with linear damping are also included in the present generalization.\\
\end{remark}

\begin{theorem}{\bf Abstract Theorem.}\label{0.5.1}
     Under the Assumptions \ref{asp1} and \ref{asp2} above, the system $(\cH, S_t)$ generated by the model \eqref{model2} in the
energy space $\cH$ is ultimately dissipative
i.e. there exists $R_0 \!> \!0$ possessing the property: 
for any bounded set B from $\cH$ there exists $t_0\! = \!t_0(B)$
such that $\|S_ty\|_\cH \!\leq \! R_0$ for all $y \in B$ and $t \!\geq \!t_0$. Moreover, there exists a forward invariant bounded absorbing set $\cB_0 \subset \cH$.
\end{theorem}
\begin{proof} 
The proof of Theorem \ref{0.5.1} is technical [given in the next section]. It is based on the so called: ``barrier's method'', and a careful construction of energy functional with the parameters determined via a suitable ``fixed'' point argument \cite{Memoir}.
In doing this, a critical fact is played by the requirement $\gamma\! < \!1 $ and the balancing condition \eqref{A3a}. This allows to iterate the  fixed point argument and to obtain sufficient a-priori and global bounds for the orbits. 
The argument, which is also inspired by \cite{haraux2,haraux} adapts the strategy presented in \cite{Memoir} and is relegated to Section \ref{ultdisabs}. Since this part is fairly technical, we shall first show how to apply the general result to the model under consideration. \\The authors would like to express thanks to Vittorino Pata  for a discussion on the subject. 

\end{proof}

\paragraph{Ultimate dissipativity for the original model governed by (\ref{sbe})-(\ref{IC}) }
\begin{theorem}\label{ball}
    The dynamical system $(\cH, S_t)$ generated by the model (\ref{sbe})-(\ref{IC}) is ultimately dissipative and possesses a forward invariant bounded absorbing set.
\end{theorem}

\begin{proof}
   The argument is based on a suitable  application of Theorem \ref{0.5.1}. This accounts to verification of the Assumptions \ref{asp1} and \ref{asp2} within the context of our model. 
Recall  that in our case; \(\cA=A, k=1 \), and the operators $ D, F, N,\Pi $ take the following form: 
          \begin{align*}
                       &D({u_t}) = g(\|u_t\|_0)u_t=b_0u_t+ \dots  + b_q \|u_t\|^q u_t ,\quad\mbox{for $s\in\R_+$, $q\geq1$, $b_j\geq0$},  b_j >0  ~for ~some ~j \in [0,q] \\
                     &F{u}=-\kappa u^+ + (\alpha-\delta\|u_x\|_0^2)u_{xx} - f(u) - \beta u_y,\quad\mbox{for every $u\in H^2_\ast$}, \\
                         &\Pi(u) = \frac{\kappa}{2}\|u^+\|_0^2-\frac{\alpha}{2}\|u_x\|_0^2+\frac{\delta}{4}\|u_x\|_0^4  + \int_\Omega \tilde{f_0}(u) \quad\mbox{for}~u\in H^2_\ast, \quad \mbox{and } N(u) = -\beta u_y.
            \end{align*}
   
In  Section \ref{thm:EU}, we have already shown that $D$ is monotone hemicontinuous and $F$ is locally Lipschitz. The assumptions on $\Pi_0$ and $\Pi_1$ have also been verified in \eqref{eq1}  and  \eqref{eq3}. The operator $N$ satisfies condition \eqref{A1-F.AST}, which follows from $|\beta| \|u_y\|_0 \leq |\beta| \|u\|_{2, \ast}$ and condition (\ref{A1}) follows from the bound imposd on function $g(s)$, so that here exist constants $c_0 \geq 0$, and $c_1 > 0$ such that
          \begin{equation*}\label{abs1}
                ||u_t||^2  \leq  c_0+c_1(Du_t,u_t), \text{ for any } u_t\in \cD(\cA ^{1/2}).
              \end{equation*}
What  is left to conclude the proof are  the last two conditions of Assumption \ref{asp2}. Lemma \ref{lemmaabs2} and Lemma \ref{L5.2} will provide the needed estimates.

\begin{lemma}\label{lemmaabs2}
 There exist constants $0\leq \eta <1$ and $c_2 > 0$; $c_3,c_4 \geq 0$ and a constant $\gamma < 1/2 $  so that the inequality (\ref{A2}) holds true for the model  \eqref{sbe} under the assumptions of Theorem \ref{0.5.1} .
 \ifdefined\xxxxx
             \begin{align}\label{abs2}
                        \mbox{so that, }-(Du_t,u)+(F^\ast(u,u_t), u) \leq &(\Pi^\prime(u),u)+\eta |A^{1/2}u|^2-c_2\Pi_0(u)\notag \\
                     &+c_3+c_4\big[1+E(u,u_t)\big]^\gamma(Du_t,u_t). 
                   \end{align} 
                   \fi
\end{lemma}
 
\begin{proof} Let $j > 0$, for any $1>\eta>0$, using {\it Holder's inequality} then \textit{Young's Inequality} for $p=\frac{j+2}{j+1}, \bar{p} =j+2,$
        \begin{align}\label{5.a}
                 b_j\|u_t\|_0^{j}|(u_t,u)| & \leq  b_j\|u_t\|_0^{j}\|u_t\|_0 \|u\|_0\!=\! b_j\|u_t\|_0^{j+1}\|u\|^{\frac{j}{j+2}}_0 \|u\|_0^{\frac{2}{j+2}}\leq C_{\eta, j}\|u_t\|_0^{j+2}\|u\|_0^{\frac{j}{j+1}}\!+\!\eta^{j+2}\|u\|_0^2 
                 \notag \\
                 \intertext{\mbox{thus for $\eta_j:=\!\eta^{j+2}$ and $\gamma_j:=\!\frac{j}{2(j+1)}$ and applying Proposition \ref{P:0} (for $\eta\!=\!1/2$) }}
                & \leq C_{\eta, j}b_j\|u_t\|_0^{j+2}\Big[1+ E(u,u_t)\Big]^{\gamma_j}+\eta_j \left\{\frac{1}{2} \cdot \left[ a(u,u) + \|u_x\|_0^4 \right]+C_{0,1/2} \right\} \notag \\
                  &\hspace*{-1.5cm}\leq C_{\eta, j}b_j\|u_t\|^{j+2}\Big[1+ E(u,u_t)\Big]^{\gamma_j}\!+\!\eta_j|a(u,u)| + \eta_j\|u_x\|_0^4+c_1, \quad c_1=\eta_jC_{0, 1/2}
        \end{align}
Note that for any $j\in [0,1...q]$, $\gamma_j \leq \gamma$ for $\gamma:=\frac{q}{2(q+1)},$ \vphantom{j<q implies that calculate it}hence applying the above inequality for each $j=0,1,...,q$, and choosing appropriate
$\eta$ so that $\sum \eta_j<\eta$ we obtain,
           \begin{align}\label{5.2}
                    -(Du_t,u)& \leq \max_q\{C_{q,\eta}\}\Big((b_0+b_1\|u_t\|^1+..+b_q\|u_t\|^{q})u_t,u_t\Big)\Big[1+ E(u,u_t)\Big]^{\gamma}\notag \\
                   & \quad \quad +(\eta_1+\dots+\eta_q)\left[|a(u,u)|+\|u_x\|_0^4\right]+c_1\notag \\
                       & \hspace*{-.5cm}\leq c_4(Du_t,u_t)\Big[1+ E(u,u_t)\Big]^{\gamma}+\eta|a(u,u)|+\eta\|u_x\|_0^4+c_1, \quad c_4=\max_q\{C_{q,\eta}\}.
             \end{align} 
Next we estimate the non-conservative term $N$, recall $F(u)=-\Pi^{'}(u)+N(u)$, thus
              \begin{equation}\label{F1}
                         \left(N(u), u\right)=(\Pi^{'} (u), u)+(F(u),u). 
               \end{equation}
               \begin{equation*}
                  \hspace{-3.9cm} \mbox{On the other hand, } |(-\beta u_y, u)| \leq |\beta|\|u\|_1\|u\|_0 \leq |\beta|\|u\|_1\|u\|_1  \leq  |\beta| \|u\|_1^2 \quad \mbox{and}
               \end{equation*}
             \begin{align}\label{F2}
                        (F(u),u)
                    &=-\kappa \|u^+\|_0^2+\alpha\|u_x\|_0^2-\delta \|u_x\|^4_0-(f_0(u),u)-\beta(u_y,u) \notag \\
                   &\leq-2\left(\frac{\kappa}{2} \|u^+\|_0^2+\frac{\delta}{4} \|u_x\|^4_0\right)+\alpha\|u_x\|^2-\frac{\delta}{2} \|u_x\|_0^4+\int_\Omega | f_0(u).u|+\beta\|u_y\|_0\|u\|_0 \notag \\
                      & \leq -2\Pi_0(u)\!-\frac{\delta}{2}\|u_x\|^4\!+\!(\alpha+\beta)\|u\|_1^2\!+\! \lambda \!\int_\Omega \! |u|^2 \quad \mbox{[here we used Assumption \ref{aspf-1}]}\notag \\
                    &\intertext{\mbox{thus by  Proposition \ref{P:0} for any  $\eta_1>0, \exists C_{1,\eta_1} $ so that }}
                        & \leq-2\Pi_0(u)-\frac{\delta}{2} \|u_x\|_0^4+\eta_1\left[\|u\|_0^2+\|u_x\|_0^4\right]+c_{1,\eta_1,\alpha,\beta}.
                \end{align}
Choosing appropriate $\eta_1$, combining \eqref{5.a}, \eqref{F1}, \eqref{F2}  the lemma  follows with $\gamma=\frac{q}{2(q+1)}<\frac{1}{2}$. 
 \end{proof}

 \begin{lemma}{[The balancing condition]}\label{L5.2}
 There exists a number $0\leq \tilde{\eta} <1$ and a non decreasing continuous function $b: \mathbb{R}_+\to \mathbb{R}_+- \{0\}$ such that 
             \begin{equation}\label{balance}
                   \lim_{s\to \infty}\Big\{s^{1-\frac{1}{\gamma}}b(s)\Big\}=0,
                \end{equation}
in the case of $\gamma>0$ where $\gamma$ is the parameter of Lemma \ref{lemmaabs2} and for any $0< \delta \leq 1$ we have the relation: $( N(u), u_t) \leq \tilde{\eta} \kappa(Du_t,u_t)+\delta E(u,u_t)+b({1/\delta}).
                $
 \end{lemma}
 \begin{proof}Using \textit{Holder's  and Young's inequalities} for $\tilde{\eta}=\frac{1}{2}(\max_j\{b_j\})$ for $p=\frac{q+1}{q+2}$ and $q=\frac{1}{q+2}$ we have,
                  \begin{align}\label{reduction1}
                         \left(N(u),u_t\right)
                         &\leq \|\beta u_y\|_0 \|u_t\|_0
                          \leq C_{\tilde{\eta}} \|\beta u_y\|_0^{\frac{q+2}{q+1}}+ \tilde{\eta} \|u_t\|^{q+2}_0\leq C_{\tilde{\eta}} \|\beta u_y\|_0^{\frac{q+2}{q+1}}+ \tilde{\eta}(Du_t,u_t)
                   \end{align}
 For the first term on the RHS, interpolation and Poincare inequality gives us, 
                   \begin{equation}\label{poincare}
                           \|u_y\|_{0}^{\frac{q+2}{q+1}} \leq  \|u\|_{0}^{\frac{q+2}{2(q+1)}}.\|\Delta u\|_0^{\frac{q+2}{2(q+1)}} \leq  C.\|u_x\|_{0}^{\frac{q+2}{2(q+1)}}.\|\Delta u\|_0^{\frac{q+2}{2(q+1)}}
                  \end{equation}
 On the last inequality above in \eqref{poincare} we have used   Poincare-type inequality, the proof of this can be found in the appendix.
 Next we estimate the RHS of \eqref{poincare} as the following by applying Young's inequality,
 {\allowdisplaybreaks
                \begin{align}\label{reduction2}
                              &C.\|u_x\|_{0}^{\frac{q+2}{2(q+1)}}.\|\Delta u\|_0^{\frac{q+2}{2(q+1)}}
                               \leq \delta \|u_x\|_{0}^4+C\delta^{\frac{-(q+2)}{7q+6}}\cdot\|\Delta u\|_0^{4.\frac{q+2}{7q+6}}\quad  \mbox{for any $\delta>0$}\quad 
                               \notag \\
                               &\leq \delta \|u_x\|_{0}^4+C \cdot \delta^{\frac{-(q+2)}{7q+6}} \cdot \epsilon\|\Delta u\|_0^{2}+C\cdot \delta^{\frac{-(q+2)}{7q+6}}\epsilon^{-\frac{2(q+2)}{5q+2}} \quad \mbox{for any $\epsilon>0$}
                               \notag \\
                            &  \leq \delta \left[ \|u_x\|_0^4+\|\Delta u\|^2_0\right]+ \delta^{-\frac{q+2}{7q+6}(1+\frac{16(q+1)}{5q+2})} \quad \mbox{ for a fix $\epsilon=\frac{1}{C}. \delta^{\frac{8(q+1)}{7q+6}}$}
                    \end{align}}
Thus combining \eqref{reduction1}, \eqref{reduction2} we get, $\left(N(u),u_t\right)\leq \tilde{\eta}((Du_t,u_t)+\delta E(u,u_t)+b\left(\frac{1}{\delta}\right)$ 
   and  $b(s)= c_{\tilde{\eta}} \cdot s^{\frac{q+2}{7q+6}\left(1+\frac{16(q+1)}{5q+2}\right)}$ for some constant $c_{\tilde{\eta}}>0$\\
To conclude the above lemma we next claim that $\lim\Big\{s^{1-\frac{1}{\gamma}}.b(s)\Big\}
     =0$ as ${s\to \infty}$.\\
    Observe that for any $q>0$,\begin{align}
     &\lim_{s\to \infty}\Big\{s^{1-\frac{1}{\gamma}}.b(s)\Big\}
     =c_{\tilde{\eta}} \cdot \lim_{s\to \infty}\Big\{s^{1-\frac{2(q+1)}{q}}.s^{\frac{q+2}{7q+6}(1+\frac{16(q+1)}{5q+2})}\Big\} 
     \leq c_{\tilde{\eta}} \cdot \lim_{s\to \infty}\Big\{s^{1-\frac{2(q+1)}{q}}\cdot s^{1+\frac{2}{q}-\frac{1}{q^2}}\Big\}=0\notag 
\end{align}

By using the property of the balancing  function $b(s)$, we then conclude that the limit is $0$. It is important to notice that ${\Tilde{\eta}}$ can be chosen such that, $c_{\Tilde{\eta}}$ is bounded and grows slower than the function ``$b$'', because the statement of the theorem asks only for the existence of such an $\Tilde{\eta}$, in our case $\Tilde{\eta}$ can be taken as $\frac{1}{2}$.
\end{proof}

\subparagraph{Continuation of the proof of Theorem \ref{ball}:}We are now in a position to conclude the existence of Absorbing ball [pending the proof of abstract result in Theorem \ref {0.5.1} to be given later]. Our model satisfies all the Assumptions in \ref{asp1} and \ref{asp2} [precisely in the above lemmas]; hence according to Theorem \ref{0.5.1} the dynamical system $(\cH, S_t)$ is ultimately dissipative. Moreover, there exists a forward absorbing set.
\end{proof}
 
Since the dynamical system $\{\cH,S_t\}$ is ultimately dissipative, it follows from Theorem 2.3.18 on page 68 of \cite{Chu15} that that there exists a weak global attractor $\mathfrak{A}$. In fact, in this case, $\mathfrak{A}$ is bounded, weakly closed, and uniformly attracts orbits/trajectory in the weak topology of $\cH$.

\noindent In the next section, we show that $\mathfrak{A}$ is indeed a ``strong'' attractor with appropriate convergence [attracting] properties in the strong topology of the phase space. .

\subsubsection{Asymptotic Smoothness.}

In order to  establish an  existence of compact global attractor for $\{\cH,S_t\}$,  it is necessary 
 to prove the  asymptotic smoothness property of the dynamical system $\{\cH,S_t\}$. The relevant result is stated below. 
\begin{theorem}\label{Asym-sm1}
        Let assumptions on $f$ and {g}  i.e. \eqref{aspf-1}, (\ref{aspg})  be in force. Then the dynamical system $(\cH, S_t)$ generated by the PDE \eqref{sbe}-\eqref{IC} is asymptotically smooth.
    \end{theorem}
    \begin{remark}
    The case of $b_0 =0$ is more demanding, since this corresponds to the degeneracy  of $g(s)$ at the origin. In this case, asymptotic smoothness follows from the ``compensated compactness'' criteria. The fact that one does not have strong control of dissipation at the origin leads to new challenges where the methods of convexity \cite{LasTat93} are employed.\\
 For $b_0 >0$, we are in a position to establish the so-called ``quasistability inequality'' (see the Definition \ref{def:QS}) in Subsection 3.1, which implies not only asymptotic smoothness, but also other properties such as the regularity of the attractor and its finite-dimensionality.
In both cases, it is shown that the ``weak attractor'' becomes a strong one as stated in Theorem \ref{Asym-sm1}.      
    \end{remark}
 
 The proof of Theorem \ref{Asym-sm1}  will follow several steps.  It is based   on a suitable application  of a compensated compactness criterion, recalled below for convenience of a reader--see \cite{Memoir}.

	\begin{proposition}[Proposition 4.13 \cite{Memoir} ]
	    \label{thm:C17}
		Assume that for any bounded forward invariant set $\cB\subset\cH$ and for any $\epsilon>0$, there exists $T=T(\epsilon,\cB)$ such that
		\vskip-.1in
  \begin{equation}\label{psinm}
			\mbox{dist}_\cH(S_T\mbf{y}_1,S_T\mbf{y}_2) 
			\leq \epsilon + \Psi_{\epsilon,\cB,T}(\mbf{y}_1,\mbf{y}_2), \quad \mbf{y}_i \in \cB
		\end{equation}
		where $\Psi_{\epsilon,\cB,T}$ is a functional on $\cB\times \cB$ such that
		\vskip-.1in
		\begin{equation}\label{proposition6.6}
			\liminf_{n\to\infty}\liminf_{m\to\infty} \Psi_{\epsilon,\cB,T}(\mbf{y}_n,\mbf{y}_m)=0,
		\end{equation}
		for every sequence $\{\mbf{y}_n\}_n\subset \cB$. Then $S_t$ is an asymptotically smooth evolution operator.
	\end{proposition}

\begin{proof} of Theorem \ref {Asym-sm1}
\paragraph{ Step 1: Preliminary Inequalities:}

\ \ In order to prove the  validity [in our case] of assumptions imposed in Proposition \ref{thm:C17}, we consider a concave, strictly increasing, continuous function $\mathcal{K}: \R+ \to \R+$, $\mathcal{K}(0) =0 $  capturing the behavior of nonlinear damping $g(|s|) s$ at the origin  (see (3.59) in \cite{Memoir}, Proposition 4.3 in \cite{eller})
\begin{equation}
\mathcal{K}\left( g(||u+v||) (u+v) - g(||u||) u , v \right)_{\Omega} \geq ||v||^2 \label{Kappa}.
\end{equation}
The construction of such a function follows from the fact that, for any strictly monotone function $ \hat{g}(s)$ satisfying $\hat{g}'(s) \geq m_1 $ for $|s|\geq 1 $, the following inequality holds:
$$ s^2 \leq m_1^{-1} [ s \hat{g}_a(s) + h (s \hat{g}_a(s)) ] , \quad s \in \R, a \in \R $$ 
where $\hat{g}_a(s) \equiv \hat{g}(s+a) - \hat{g}(a) $, and $h(s)$ is  a monotone, concave, and continuous function with $h(0) =0 $. See page 134 \cite{Memoir}, Proposition 4.3 \cite{eller} and related references \cite{LasTat93}. The above construction can be applied within our context with $\hat{g}(s)= g(|s|)s$.
In fact, in the present specific case, the function $h(s) \sim s^{\frac{2}{j+2}} $ captures the behavior of the damping at the origin with $ j \geq 1 $, such that $b_j$ is the first term in the damping formula which is different from zero. 
We also need several inequalities related to the energy to be used  for the development.  

 Let $\mbf{u}=\{u,u_t\}$ and $\mbf{v}=\{v,v_t\}$ be strong solutions of \eqref{sbe}-\eqref{IC} uniformly bounded in time, i.e. $\|\mbf{u}(t)\|_\cH,\|\mbf{v}(t)\|_\cH\leq R$ for every $t\geq0$ and for some $R>0$. Then $z=u-v$ is the (strong) solution and it  solves the following:
    \begin{equation}\label{1abs-z}
        z_{tt} + Az + D(u_t)-D(v_t) = F(u)-F(v) ,
    \end{equation}
           with initial data $\mbf{z}_0=\mbf{u}_0-\mbf{v}_0$.  $\mbf{z}=\{z,z_t\}$ satisfies the energy identity with $\tilde{E}(\mbf{z(t)}) \equiv ||\mbf{z(t)}||^2_{\cH} $.
    \begin{equation}\label{1lem:EE-01}
        \tilde{E}(\mbf{z}(t)) + \int_s^t(D(u_t)-D(v_t),z_t)_\Omega
        = \tilde{E}(\mbf{z}(s)) + \int_s^t(F(u)-F(v),z_t)_\Omega,\quad s\leq t.
    \end{equation}
    Since $(D(u_t)-D(v_t),z_t)_\Omega = \sum_{j=0}^{q}b_j(\|u_t\|_0^ju_t-\|v_t\|_0^jv_t,z_t)_\Omega \geq 0$ (see Lemma \ref{lem:op_D}).
    \begin{equation}\label{1.lem:EE-01}
        \tilde{E}(\mbf{z}(t))
        \leq \tilde{E}(\mbf{z}(s)) + \int_s^t(F(u)-F(v),z_t)_\Omega,\quad s\leq t.
    \end{equation}
          Let $T>0$  be arbitrary. Applying \eqref{1.lem:EE-01} for $\{s,t\}=\{t,T\}$ and integrating over $[0,T]$, we obtain
    \begin{align}\label{1lem:EE-02}
        T\tilde{E}(\mbf{z}(T)) 
        \leq \int_0^T\tilde{E}(\mbf{z})
        + \int_0^T\!\!\int_t^T(F(u)-F(v),z_t)_\Omega .
    \end{align}
 Now, multiplying \eqref{1abs-z} by $z$ and integrating by parts on $(0,T)\times\Omega$, we obtain the \textit{\bf recovery  of the  energy}
    \begin{align}\label{1lem:EE-04}
        2\!\!\int_0^T\! \!\tilde{E}(\mbf{z})
        \!\leq\! \tilde{E}(\mbf{z}(0))\!+\!\tilde{E}(\mbf{z}(T))\!+\! &2\!\! \int_0^T \!\!\|z_t\|_0^2
        \!+\!\!\int_0^T\!\! | (D(u_t)-D(v_t),z)_\Omega|\!+\! \int_0^T|(F(u)-F(v),z)_\Omega|,
    \end{align}
combining \eqref{1lem:EE-02} and \eqref{1lem:EE-04} yields:
    \begin{align}
        T\tilde{E}(\mbf{z}(T)) + \int_0^T \tilde{E}(\mbf{z})
        \leq \int_0^T 2\|z_t\|_0^2+\tilde{E}(\mbf{z}(0))+\tilde{E}(\mbf{z}(T))
        \!\!+\! \int_0^T\!\! |(D(u_t)-D(v_t),z)_\Omega|  \notag\\
        \!+ \!\int_0^T\!\!|(Fu-Fv,z)_\Omega|
        \!+ \!\int_0^T\!\!\int_t^T(F(u)-F(v),z_t)_\Omega, \label{1lem:EE-03}
    \end{align}
using the energy identity \eqref{1.lem:EE-01} we have
        $$\tilde{E}(\mbf{z}(0)) =  \tilde{E}(\mbf{z}(T))+ \int_0^T(D(u_t)-D(v_t),z_t)_\Omega
        - \int_0^T(F(u)-F(v),z_t)_\Omega,$$
taking $T \geq 2+ \frac{T}{2}$ and combining the above inequality with  \eqref{1lem:EE-03} we have the recovery inequality
     \begin{align}
         \frac{T}{2}\tilde{E}(\mbf{z}(T)) &+ \int_0^T 
              \tilde{E}(\mbf{z})
             \leq (T-2)\tilde{E}(\mbf{z}(T)) + \int_0^T \tilde{E}(\mbf{z})\notag \\
           &  \leq \int_0^T 2\|z_t\|_0^2
             + \int_0^T (D(u_t)\!-\!D(v_t),z_t)_\Omega  + \int_0^T |(D(u_t)-D(v_t),z)_\Omega| \notag\\
           &- \!\!\int_0^T\!\!(F(u)\!-\!F(v),z_t)_\Omega\!+\!\int_0^T\!\!| 
                (F(u)- 
              F(v),z)_\Omega|
              + \int_0^T\!\!\int_t^T\!\!(F(u)-F(v),z_t)_\Omega \notag \\
         &\leq \!\int_0^T 2\|z_t\|_0^2
              \!+\! \int_0^T \!\!(D(u_t)\!-\!D(v_t),z_t)_\Omega \!+\! \int_0^T \! |(D(u_t)-D(v_t),z)_\Omega|+ \psi(\mbf{z, T})\label{13lem:EE-03} \\
        \text{where, } \psi(\mbf{z, T})= \!-\!& 
               \int_0^T\!\!(F(u)\!-\!F(v),z_t)_\Omega+\!\!\int_0^T\!\!| 
                (F(u)\!-\! 
              F(v),z)_\Omega|
              + \int_0^T\!\!\int_t^T(F(u)-F(v),z_t)_\Omega. \label{14lem:EE-03}
    \end{align}

{\bf Step 2. Supporting Lemmas}. In the process of the proof of Theorem \ref{Asym-sm1}  we will take $T$ as a large number for the inequalities in \eqref{14lem:EE-03} in the interval $[0,T]$, so the dependence on $T$ is critical and needs to be controlled. In order to prove the Theorem, we need  several additional  lemmas.

    \begin{lemma}\label{damping1}
    Let $\mbf{u}=\{u,u_t\}$ and $\mbf{v}=\{v,v_t\}$ be strong solutions of \eqref{sbe}-\eqref{IC} corresponding to the initial data $\mbf{y}_1=(u_0,u_1)$ and $\mbf{y}_2=(v_0,v_1)$ from the bounded set $\cB$ and $\mbf{z}=\mbf{u}-\mbf{v}$. Then the following inequality hold:
    \begin{itemize}
        \item there exists  a concave, strictly increasing, continuous function $\mathcal{K}$ with the property given in \eqref{Kappa} such that 
          \begin{equation}\label{ut2}
              \int_0^T \|z_t\|^2 dt \leq T \mathcal{K}\left(\frac{1}{ T} 
                  \int_0^T(D(u_t)-D(v_t), z_t)dt\right).
          \end{equation}
    \end{itemize}
    \end{lemma}
    \begin{proof} 
        Using the property \eqref{Kappa}, we have 
        \begin{align}
             & \int_\Omega |z_t|^2 
             \leq 
             \mathcal{K}\left( D(u_t) -D(v_t),z_t\right)_{\Omega},\notag \\
             \intertext{\mbox{by integrating  in time and applying Jensen's inequality to obtain,}}
    &\int_0^T \|z_t\|^2 dt \leq 
             \int_0^T \mathcal{K} \left[( D(u_t) -D(v_t),z_t)_{\Omega}
            \right ]dt \leq T 
            \mathcal{K}\left( 
            \frac{1}{T }
              \int_0^T  
           (D(u_t) -D(v_t), z_t)_{\Omega}dt \right)\notag.
        \end{align}
    Thus, the proof is completed.  
    \end{proof}

   \begin{lemma}\label{Dutu.}
       For some $2 \geq \eta>0, t \in [0, T]$  and  $\mbf{u}, \mbf{v}$ as described in \eqref{damping1} one obtains:
       \begin{equation*}
           \int_0^T |(D(u_t)-D(v_t),z)_\Omega| \leq \int_0^T ||z_t||^2_0   + C_{\cB} T \sup_{t \in [0, T]} \|z\|_{2-\eta, \Omega}^2.
       \end{equation*}
   \end{lemma}
    \begin{proof}
    Using the same strategy as before in \eqref{lem:EE-06}, one obtains,
    \begin{align}
        \int_0^T|(D(u_t)-D(v_t),z)_\Omega|\leq \!
            &\int_0^T\!\!\|z_t\|_0^2 
               \!+C_{\cB, q} T \sup_{t \in [0, T]} \|z\|_0^2. 
            \notag
    \end{align}
    Thus, the lemma is proved for $\eta=2$.
    \end{proof}
    
    \begin{lemma}\label{fz}
     For some $\gamma>0, t \in [0, T]$ and $\mbf{u}, \mbf{v}$ as described in \eqref{damping1} we have,
     \begin{equation}\label{1fz}
         \int_0^T |(F(u)-F(v),z)| \leq \gamma \int_0^T \tilde{E}(\mbf{z})+C_{\cB, \gamma}\int_0^T\|z\|_0^2. 
     \end{equation}
    \end{lemma}
    \begin{proof}
    Since $F$ is locally Lipschitz from $H^2_\ast$ (by \eqref{lem:op_F}) to $L^2(\Omega)$, we have
    \begin{align*}
        \abs{\int_0^T(F(u)-F(v),z)_\Omega}
        \leq L_{\cB}\int_0^T\!\!\|z\|_{2,\ast}\|z\|_0
        \leq \gamma\int_0^T\!\!\tilde{E}(\mbf{z})
        + \frac{L_{\cB}^2}{4\gamma}\int_0^T\|z\|_0^2,
    \end{align*}
    for every $\gamma>0$ (arbitrarily small), where $L_B>0$ is the Lipschitz constant of $F$ on $B$.   
     \end{proof}
     
     {\bf Step 3: Verification of  the assumptions postulated by  Proposition \ref{thm:C17}.}
     Observe that for strong solutions $\mbf{u}, \mbf{v} $ and $\mbf{z}=\mbf{u}-\mbf{v}$ the relation \eqref{1lem:EE-01} and existence of absorbing ball ($\cB$) gives us $\tilde{E}(\mbf{z}(t)) \leq C_{\cB}$ for any $t \geq 0$, hence
     \begin{align}\label{12lem:EE-01}
         \int_s^t(D(u_t)-D(v_t),z_t)_\Omega- \int_s^t(F(u)-F(v),z_t)_\Omega
        \leq \abs{ \tilde{E}(\mbf{z}(s))-\tilde{E}(\mbf{z}(t))}\leq 2 C_{\cB} ,\quad s\leq t.
     \end{align}
     Using the Lemmas \ref{damping1}, \ref{Dutu.} and \ref{fz} in \eqref{13lem:EE-03}, we obtain with any $\gamma >0$,
     \begin{align}
         \frac{T}{2}\tilde{E}(\mbf{z}(T)) 
              &\!+\! \int_0^T 
              \tilde{E}(\mbf{z}) \!\leq \!4T C_{\Omega}\left[\mathcal{K}\left(\frac{1}{ T  }\int_0^T(Du_t- Dv_t,z_t)_{\Omega)})dt\right)\right]  \notag \\
              &  +\! \int_0^T \!(D(u_t)-D(v_t),z_t)_\Omega \!+\!C_{\cB, q} T 
                  \sup_{t \in [0, T]} \|z\|_0^2+ \gamma \int_0^T \tilde{E}(\mbf{z})+C_{\cB, \gamma}\int_0^T\|z\|^2 \notag \\
              &  -\int_0^T(F(u)-F(v),z_t)_\Omega
                        + \int_0^T\!\!\int_t^T(F(u)-F(v),z_t)_\Omega\notag \\
              &   \leq \!4T C_{\Omega}\left[\mathcal{K}\left(\frac{1}{ T  
                  }\int_0^T(D(u_t)- D(v_t),z_t)_{\Omega} dt\right)\right]+2C_{\cB} \notag \\
              \hspace{-1pt}&  +\!\left(C_{\cB, q}+C_{\cB, \gamma}\right) T 
                  \sup_{t \in [0, T]} \|z\|_0^2+ \gamma \int_0^T \tilde{E} 
                 (\mbf{z})+ \abs{\int_0^T\!\!\int_t^T(F(u)-F(v),z_t)_\Omega}\notag \\
                &\intertext{\mbox{here we used \eqref{12lem:EE-01}, finally for  $\gamma<1,$ we have, }}
                    \frac{T}{2}\tilde{E}(\mbf{z}(T)) 
              &   \leq \!4T C_{\Omega}\left[\mathcal{K}\left(\frac{1}{ T  
                  }\int_0^T(D(u_t) -D(v_t),z_t)_{\Omega}dt\right)\right]+2C_{\cB} \notag \\
              &  \hspace{1cm}+\!C_{\cB, q} T 
                  \sup_{t \in [0, T]} \|z\|_0^2+ \abs{\int_0^T\!\!\int_t^T(F(u)-F(v),z_t)_\Omega}, \\
             \intertext{\mbox{however, energy relation for  $z$ variable which eliminates the damping reduces this to, } }
              &\hspace{-1.5cm}\tilde{E}(\mbf{z}(T)) \leq 
                  4 C_{\Omega}\mathcal{K}\left(\frac{1}{ T } 
                   C_B\right)+\mathcal{K}\left( \abs{\frac{1}{T}  \int_0^T ( F(u) - F(v), z_t)_{\Omega}}\right)+\frac{2C_{\cB}}{T}\notag \\
              &  +\!C_{\cB, q}
                  \sup_{t \in [0, T]} \|z\|_0^2+ \frac{1}{T}\abs{\int_0^T\!\!\int_t^T(F(u)-F(v),z_t)_\Omega,}\label{id2}
     \end{align}

     We have used the estimate resulting from energy relation and concavity as well as monotone property of $\mathcal{K}$ as stated in the following inequality,  
     $$\mathcal{K}\left( \frac{1}{T} \int_0^T (D(u_t) - D(v_t),z_t)_{\Omega}\right) \leq \mathcal{K} 
               \left(\frac{1}{T} C_B\right) + \mathcal{K}\left( \abs{\frac{1}{T}  \int_0^T ( F(u) - F(v), z_t)_{\Omega}}\right). $$
Recall that the function $\mathcal{K}$ is monotone increasing and concave, thus we have for any $\epsilon_0>0$ the estimate
 \begin{equation}\label{K_2}
     \mathcal{K}\left( \abs{\frac{1}{T}  \int_0^T ( F(u) - F(v), z_t)_{\Omega}}\right) \leq \epsilon_0+C_{\epsilon_0} \abs{\frac{1}{T}  \int_0^T ( F(u) - F(v), z_t)_{\Omega}}.
 \end{equation}
Using this relation we get from \eqref{id2},
 \begin{align}
     \tilde{E}(\mbf{z}(T)) \leq 
                  \frac{2C_{\cB}}{T}
            &+4      C_{\Omega}\mathcal{K}\left(\frac{1}{ T } 
                   C_B\right)+ \epsilon_0+C_{\epsilon_0} \abs{\frac{1}{T}  \int_0^T ( F(u) - F(v), z_t)_{\Omega}}\notag \\
            &+\!C_{\cB, q}
                  \sup_{t \in [0, T]} \|z\|_0^2+ \frac{1}{T}\abs{\int_0^T\!\!\int_t^T(F(u)-F(v),z_t)_\Omega}.\label{Ene}
 \end{align}
We call the function $\Psi$ as the following:
 \begin{align}
     &\Psi_{T, \cB}={\Psi^1}_{T,\cB}\left(\mbf{y}_1,\mbf{y}_2\right) + C_{\epsilon_0}\abs{\frac{1}{T}  \int_0^T ( F(u)- F(v), z_t)_{\Omega}},\label{Psi03} \\
     & \mbox{where, } {\Psi^1}_{T,\cB}=\!C_{\cB, q}\sup_{t \in [0, T]} \|z\|_0^2+ \frac{1}{T}\abs{\int_0^T\!\!\int_t^T(Fu-Fv,z_t)_\Omega}.
 \end{align}
 Note that it is sufficient to prove the {\it lim inf} condition for $\Psi^1$ function because the second term of $\Psi$ in \eqref{Psi03} will satisfy the {\it lim inf} condition by using the same argument that we will be using for $\Psi^1$. The $\epsilon$ in the formulation of the Proposition  \ref{thm:C17} 
 will be obtained by taking  $T$ large enough so that $\epsilon \equiv \epsilon_0 + \frac{2C_B}{T} $ 
 can be made arbitrarily small. 
 The reminder of the the proof will be concentrated on prove the liminf condition $\Psi^1_{T,B}.$
 
 \noindent Recall from Section \ref{sec:Pre} that, the nonlinear operator $F$ has the form,
   $$Fu =-\left[\kappa u^+ + (\alpha-\delta\|u_x\|_0^2)u_{xx} + f_0(u) + \beta u_y\right].$$
 We consider a sequence of initial data $\{y_n\} \in \cB$. Since  the energy is uniformly  bounded, the corresponding solution $w_n(t)$ of the initial data $y_n$ satisfy the following on a subsequence:
     \begin{align}
     & w^n \to w \mbox{ weakly star in } L^{\infty}(0,T; H^2_\ast(\Omega)),\label{weakast1}\\
     &(w^n)_t \to w_t \mbox{ weakly star in } L^{\infty}(0,T; L^2(\Omega)).\label{weakast2}
     \end{align}
However tracing back in time, from these solutions $w^n(t)$ we can reconstruct the corresponding subsequence of initial data. On the other hand, the relation \eqref{psinm} holds if there exists a subsequence $\{y_{n_k}\}_k \subset \{{y_n}_n\}$ such that,
      $$\lim_{k \to \infty}\lim_{l \to \infty}{\Psi^1}_{T, \cB}(y_{n_l}, y_{n_k})=0.$$ 
Thus it is sufficient to prove that for $\mbf{y}_n=\{w^n, w^n_t\}$ and  $\mbf{y}_m=\{w^m, w^m_t\}$,
     \begin{align}
          &\hspace{2cm}\liminf_{n \to \infty} \liminf_{m \to \infty}{\Psi^1}_{\epsilon,T, \cB}\left({\mbf{y}_{n},\mbf{y}_{m}}\right)=0, \label{psi1}
    \end{align}
   for every fixed $T$ and for any sequence satisfying the properties \eqref{weakast1}, \eqref{weakast2}, where
    \begin{align}
        {\Psi^1}_{\epsilon,T, 
\cB}\left({\mbf{y}_{n},\mbf{y}_{m}}\right) = C_{\cB, q}\sup_{t \in [0, T]} \|w^n-w^m\|_0^2+\frac{1}{T}\abs{\int_0^T\!\!\int_t^T\kappa((w^n)^+-(w^m)^+,w^n_t-w^m_t)_\Omega}, \notag \\
          \quad +\frac{1}{T}\abs{ 
  \int_0^T\!\!\int_t^T\left((\alpha -\delta \|w^n_x\|^2)w^n_{xx}-(\alpha-\delta\|w_x^m\|^2w_{xx}^m),w^n_t-w^m_t\right)} \notag \\
          \quad +\frac{1}{T}\abs{\int_0^T\!\!\int_t^T(f_0(w^n)-f_0(w^m),w^n_t-w^m_t)_\Omega} +\frac{1}{T}\abs{\int_0^T\!\!\int_t^T\beta(w_y^n-w_y^m,w^n_t-w^m_t)_\Omega}. \label{subseq}
     \end{align}
Note that the first factor in the RHS of \eqref{subseq} is
a ``lower order term'', for which the liminf-condition is automatically verified. This follows from Aubin-Lions compactness result.

\noindent For the second term, we essentially deal with the product of weakly and strongly convergent sequences. The details are below.
\begin{align}
&\int_t^T\int_\Omega  \kappa ((w^n)^+-(w^m)^+)\cdot (w^n_t-w^m_t)\notag \\
&=\kappa \frac{1}{2} \left[\|(w^n)^+\|^2+\|(w^m)^+\|^2\right]_t^T-\kappa \int_\Omega \int_t^T ((w^n)^+ , w^m_t)+((w^m)^+ ,w^n_t)
\end{align}
the compactness of $H^2_\ast$ in $L^2$ gives us,
               \begin{equation}\label{same}
                      \left . \left . \lim_{n \to \infty} \|(w^n)^+\|^2\right|_t^T=\|w^+\|^2\right|_t^T \mbox{ and }\left . \left . \lim_{m \to \infty} \|(w^n)^+\|^2 \right|_t^T=\|w^+\|^2\right|_t^T.
                \end{equation}
Due to the energy bounds, the solution $w_n(t)$ satisfies (on a subsequence):
                 \begin{align}
                       w_{t}^n \to w_t \mbox{ weakly star in } L_\infty(0, T; L^2(\Omega)),\notag
                 \end{align}
                   \begin{equation}\label{diff}
                        \mbox{hence, }\lim_{n \to \infty} \lim_{m \to \infty} \int_\Omega \int_t^T ((w^n)^+ \cdot w^m_t+(w^m)^+ \cdot w^n_t)=\left .\|w^+\|^2\right|_t^T,
                   \end{equation}
combining \eqref{same} and \eqref{diff} we conclude that,
 \begin{equation}\label{kappa}
     \liminf_{n \to \infty} \liminf_{m \to \infty}\int_t^T\int_\Omega  \kappa ((w^n)^+-(w^m)^+)\cdot (w^n_t-w^m_t)=0.
 \end{equation}
 For the third term in the expression of ${\Psi^1}$ note that
 $\mbf{y}_n=(w^n,w^n_t)$ is a strong solution thus for $w_t\in \cH_2^*$ integration by parts yields,

 \begin{align}
      & \int_t^T \int_\Omega(\alpha - 
          \delta \|w^n_x\|^2)w^n_{xx}w^n_t= -\left .\frac{\alpha}{2}\|w^n_x\|_0^2\right|_t^T + \left .\frac{\delta}{4}\|w^n_x\|_0^4\right |_t^T,\notag \\
   \intertext{\mbox{therefore by compactness $\|w^n_x\|^2 \to \|w_x\|^2$ and $\|w^n_x\|^4 \to \|w_x\|^4$, we have,}}
        & \lim_{n \to \infty} \lim_{m \to \infty} \int_t^T \int_\Omega(\alpha - 
          \delta \|w^n_x\|^2)w^n_{xx}w^n_t=\left .-\frac{\alpha}{2}\|w_x\|_0^2\right|_t^T + \left .\frac{\delta}{4}\|w_x\|_0^4\right |_t^T, \notag \\
\intertext{\mbox{by the similar argument we conclude}}
            &   \lim_{n \to \infty} \lim_{m \to \infty}\! \int_t^T \!\!\int_\Omega\!\!\left[(\alpha\! -\! 
              \delta \|w^n_x\|^2)w^n_{xx}w^n_t\!+\!(\alpha \!-\! 
          \delta \|w^m_x\|^2)w^m_{xx}w^m_t\right]\!=\!\left .\!- \!
              {\alpha}\|w_x\|_0^2\right|_t^T \!+\! \left .\frac{\delta}{2}\|w_x\|_0^4\right |_t^T. \label{alphadel2}
  \end{align}

\noindent For the mixed [both $n$ an $m$] terms in the third term, the iterated limits give, 
                \begin{align}
                      &\lim_{n \to \infty} 
                          \lim_{m \to \infty}\int_t^T \int_\Omega(\alpha - 
                         \delta \|w^n_x\|^2)w^n_{xx}w^m_t=-\lim_{n \to \infty} 
                         \lim_{m \to \infty} \int_t^T \int_\Omega(\alpha - 
                           \delta \|w^n_x\|^2)w^n_{x}w_{xt}^m\notag \\
                         &=-\lim_{m \to \infty} \int_t^T(\alpha - 
                            \delta \|w_x\|^2)\int_\Omega  w_{x}w^m_{xt} =  -\left.\frac{\alpha}{2}\|w_x\|^2+\frac{\delta}{4} \|w_x\|^4 \right |_t^T.\notag
                        \end{align}
 By similar argument we achieve,
 \begin{align}\label{alpha6}
     \lim_{n \to \infty} 
         \lim_{m \to \infty} \!\int_t^T\!\!\int_\Omega\!\big[(\alpha\! -\! 
                         \delta \|w^n_x\|^2)w^n_{xx}w^m_t\!+\!(\alpha\! -\! 
          \delta \|w^m_x\|^2)w^m_{x}w_{xt}^n\big]=- \left.\alpha\|w_x\|^2+\frac{\delta}{2} \|w_x\|^4 \right |_t^T.
 \end{align}
 Combining \eqref{alphadel2} with \eqref{alpha6} we conclude  that,
 \begin{align}\label{adf}
     \liminf_{n \to \infty} 
         \liminf_{m \to \infty} \int_t^T\int_\Omega \left((\alpha -\delta \|w^n_x\|^2)w^n_{xx}-(\alpha-\delta\|w_x^m\|^2)w_{xx}^m)\cdot (w^n_t-w^m_t\right)=0.
 \end{align}
 Next we focus on the fourth term of $\Psi^1$, we claim, 
 \begin{equation}\label{flim}
     \liminf_{n \to \infty} \liminf_{m \to \infty}\int_0^T\int_t^T(f_0(w^n)-f_0(w^m),w^n_t-w^m_t)_\Omega=0.
 \end{equation}
The proof will be similar to how we have handled the second term in ${\Psi^1}$ i.e; by grouping the mixed items together, let $\tilde{f}$ is the antiderivative of $f$ then
           \begin{align}\label{lipf}
                      \int_t^T\! \!\int_\Omega\!(f_0(w^n)\!-\!f_0 (w^m))(w_t^n-w_t^m)  = \int_{\Omega} [ \tilde{f}_0 (w^m) + \tilde{f}_0(w^m) ]|_t^T \!-\!\! \int_t^T\!\!\int_\Omega \left[f_0(w^n) w_t^m\!+\!f_0(w^m) w_t^n\right].
           \end{align}
For the first two terms using MVT and Young's inequality, we have,
\begin{align}\label{wsc-f1-2}
			&\abs{\int_\Omega\tilde{f}(w^n) - 
                  \int_\Omega\tilde{f}(w)} 
                  \leq C_{T,\mathcal{B}}|\Omega|^{1/2} \int_\Omega |w^n-w|
                  \leq C_{T,\mathcal{B}}|\Omega|^{1/2} \|w^n-w\|_0,
	\end{align}
	where $C_{T,\cB}>0$.
	It follows from the compact embedding and inequality \eqref{wsc-f1-2} that
	\vskip-.15in
	\begin{align}
		&\lim_{n\to\infty}\left.\int_\Omega\tilde{f_0}(w^n)\right|_0^T
		= \left.\int_\Omega\tilde{f}_0(w)\right|_0^T \quad \mbox{and }\notag \\
		&\lim_{n\to\infty}\liminf_{m\to\infty}\left\{
		\left.\int_\Omega \tilde{f_0}(w^n)\right|_0^T
		+ \left.\int_\Omega \tilde{f_0}(w^m)\right|_0^T
		\right\}
		= 2\left.\int_\Gamma\tilde{f_0}(w)\right|_0^T, \label{wsc-f1-3}
	\end{align}
for the mixed terms in the RHS of \eqref{lipf}, we have $w^m_t\to w_t$ weakly in $L^2(Q_0^T)$, on the other hand, since $w^n\to w$ almost everywhere in $Q_0^T$, it follows from the continuity of $f_0$ that $f_0(w^n)\to f(w)$ almost everywhere in $Q_0^T$. Therefore, we must have
	$f(w^n)\to f(w)$ weakly in $L^2(Q_0^T)$. Thus after taking the iterated limits,
	we have,
	\vskip-.15in
	\begin{equation}\label{wsc-f0-4}
		\liminf_{n\to\infty}\liminf_{m\to\infty}
		\int_{\Omega}\int_0^T\left[f_0(w^n)w^m_t+f_0(w^m)w^n_t\right]
		= 2\left.\int_\Omega\tilde{f_0}(w)\right|_0^T,
	\end{equation}
 thus we get \eqref{flim} by combining \eqref{wsc-f0-4} and \eqref{wsc-f1-3}.\\
For the last term, in ${\Psi^1}$ we deal  again with the product of weak and strong convergence. 
 This can be seen directly  from the calculations below. 
 \begin{align}
    &\int_t^T \int_\Omega\beta(w_y^n-w_y^m) \cdot (w^n_t-w^m_t)=\beta \int_t^T\int_\Omega(w^n_y w_t^n+w_y^mw_t^m)-\beta \int_t^T\int_\Omega (w^n_y w^m_t+w_y^mw_t^n)\notag \\
  &\lim_{n \to 
       \infty}\lim_{m \to 
       \infty}\int_t^T\int_\Omega (w^n_y w_t^n+w_y^m w_t^m)= 2 \int_t^T\int_\Omega w_y w_t, \notag \\
  \intertext{\mbox{for the mixed terms we use the 
         weak convergence of $w_t$ in $L^2$ and strong convergence, }}
  &     
 \mbox{ of $w_y$ in $L^2$ to conclude }\lim_{n\to\infty}\lim_{m\to\infty}
		\int_{\Omega}\int_0^T \! \!(w^n_y w^m_t+w_y^m w_t^n)=2 \int_\Omega \int_0^T w_y w_t. \notag
 \end{align}
The above equalities then give us,
\begin{equation}\label{5}
   \lim_{n\to\infty}\lim_{m\to\infty} \int_t^T \int_\Omega\beta(w_y^n-w_y^m) \cdot (w^n_t-w^m_t)=0.
\end{equation}
Therefore we have shown that all the items \eqref{kappa} \eqref{adf}, \eqref{flim}, \eqref{5} on the right hand side of \eqref{subseq} goes to $0$ individually when taking the limits on the subsequences. This proves \eqref{subseq}.
Now for any given $\epsilon>0$, choosing sufficiently large $T$ in \eqref{id2} and choosing $\epsilon_0< \epsilon$ we have the Proposition \ref{thm:C17}  proved since the function $\mathcal{K}$ is continuous and zero at the origin. Thus we have, for any $\epsilon>0$
\begin{equation}\label{psinnm}
			\mbox{dist}_\cH(S_T\mbf{y}_1,S_T\mbf{y}_2) 
			\leq \epsilon + \Psi_{\epsilon,\cB,T}(\mbf{y}_1,\mbf{y}_2), \quad \mbf{y}_i \in \cB.
		\end{equation}
The proof of Theorem \ref{Asym-sm1} is completed by evoking Proposition \ref{thm:C17}. 
\end{proof}

\noindent We are now in a position to establish the existence of global attractors in both degenerate and non-degenerate cases-as stated in {\bf the first part of Theorem \ref{thm:Att}}. Indeed, this follows by combining the results of Theorem \ref{Asym-sm1} with Corollary \ref{ball} and the following known result:  

\begin{corollary}{[see \cite{Memoir}, page 18]}
    Let $(\cH, S_t)$ be an ultimately dissipative dynamical system in a complete metric space $\cH$. Then $(\cH, S_t)$ possesses a compact global $\mathfrak{A}$ if and only if $(\cH, S_t)$ is asymptotically smooth.  
\end{corollary}
\subsection{Proof of Theorem \ref{thm:Att}--Part II}\label{attractor11}
The next step is to prove {the second part of Theorem \ref{thm:Att}} under the stronger assumption $ g(0) =  b_0 >0$. In that case, we will be able to establish a stronger ``squeezing'': type of property referred to as {\it quasistability.} Quasi-stability is a powerful tool  used for characterization of   additional properties of attractors such as smoothness, dimensionality and  exponential attraction \cite{Memoir,LasMaMo}.


When $g(0)$ is positive i.e., $b_0>0$ in \eqref{aspg}, we have additional information regarding the geometry of the attractor $\mathfrak{A}$ which will be discussed in the next section.

\ifdefined\xxxxx
\textcolor{red}{
\begin{definition}\label{def:QS}
    The dynamical system $\{\cH,S_t\}$ is said to be \textbf{quasi-stable} on a set $B\subset\cH$, if there exists $t_\ast>0$, a Banach space $Z$, a globally Lipschitz mapping $K:B\to Z$ and a compact seminorm $n_Z$ on $Z$, such that for every $\mbf{y}_1,\mbf{y}_2\in B$ with $0\leq p<1$, we have
                \begin{equation*}
                           \|S_{t_\ast}\mbf{y}_1-S_{t_\ast}\mbf{y}_2\|_\cH
                          \leq p\|\mbf{y}_1-\mbf{y}_2\|_\cH + \textrm{n}_Z(K\mbf{y}_1-K\mbf{y}_2).
                \end{equation*}
\end{definition}
Quasi-stability is a powerful tool  used for characterization of   additional properties of attractors such as smoothness, dimensionality and  exponential attraction \cite{Memoir,LasMaMo}.}
 \fi
 \subsubsection{Quasi-stability.}
 Quasistability of a related model with a  {\it linear} damping was previously shown in \cite{BGLW22}.
According to Proposition 3.4.17 on \cite{Chu15}, the quasi-stability property can be obtained via {\it asymptotic quasi-stability}, defined in the Section \ref{mainresults}. See also  \cite{yuming,Mato,Mato1} for applications of this property to several different models. .
\begin{proposition}\label{pro:QS}
Under the assumptions of Theorem \ref{thm:Att},  with $g(0) > 0 $, the dynamical system $\{\cH,S_t\}$ is quasi-stable on every bounded forward invariant set of $\cH$.
\end{proposition}
\noindent The proof of Proposition \ref{pro:QS} is based on several lemmas.
\begin{proof}
\ifdefined\xxxxxx
\textcolor{red}{
\begin{definition}
    $\{\cH,S_t\}$ is said to be \textbf{asymptotically quasi-stable} on a set $B\subset\cH$ if there exists a compact seminorm $\mu_{H^2_\ast}$ on $H^2_\ast$ and non-negative functions $h(t)$, $g_0(t)$ and $d(t)$ on $\R_+$ such that: [i.] $h(t)$ and $d(t)$ are locally bounded on $[0,\infty)$, [ii.] $g_0(t)\in L^1(\R_+)$ possess the property $\lim_{t\to\infty}g_0(t)=0$, and [iii.] for every $\mbf{y}_1,\mbf{y}_2\in B$ and $t>0$ the following holds
                 \begin{equation*}
                                \|S_t\mbf{y}_1-S_t\mbf{y}_2\|_\cH^2 \leq h(t)\|\mbf{y}_1-\mbf{y}_2\|_\cH^2 \quad \mbox{and}
                 \end{equation*}
                \begin{align*}
                              \|S_t\mbf{y}_1-S_t\mbf{y}_2\|_\cH^2
                                \leq g_0(t)\|\mbf{y}_1-\mbf{y}_2\|_\cH^2
                               + d(t)\sup_{0\leq s\leq t}\left[\mu_{H^2_\ast}(y^1(s)-y^2(s))\right]^2,
                   \end{align*}
    where $S_t\mbf{y}_i=\{y^i(t),y^i_t(t)\}$, $i=1,2$, are the corresponding solutions of \eqref{sbe}-\eqref{IC}.
\end{definition}}
\fi

{\bf Step 1:} [Energy recovery Estimate.] 
We look at the difference of trajectories of two solutions. The following estimate holds. 
\begin{lemma}\label{lem:EE}
    Let $T>0$ arbitrarily fixed. Let also $\mbf{u}=\{u,u_t\}$ and $\mbf{v}=\{v,v_t\}$ be strong solutions of \eqref{sbe}-\eqref{IC} uniformly bounded in time, i.e. $\|\mbf{u}(t)\|_\cH,\|\mbf{v}(t)\|_\cH\leq R$ for every $t\geq0$ and for some $R>0$. Therefore, if $\mbf{z}=\mbf{u}-\mbf{v}$ then there are positive constants $C$ [independent on T]  and $C_{R,T}$ such that
    \begin{align}
        T\tilde{E}(\mbf{z}(T)) 
        + \int_0^T\!\!\tilde{E}(\mbf{z})
        \leq 
        C\left[\tilde{E}(\mbf{z}(0)) 
        + \int_0^T\!\!(F(u)-F(v),z)_\Omega
        + \int_0^T\!\!(F(u)-F(v),z_t)_\Omega \right. \notag\\ \left.
        + \int_0^T\!\!\int_t^T\!\!(F(u)-F(v),z_t)_\Omega \right]
        + C_{R,T}\lot{[0,T]}{z}, \label{lem:EE-ineq}
    \end{align}

    \begin{equation*}\label{lem:EE-lot}
       \mbox{where, }\quad \lot{[0,T]}{z} = \sup_{t\in[0,T]}\left\{\|z(t)\|_{2-\eta}^2:\eta\in(0,2]\right\}.
    \end{equation*}
\end{lemma}

\begin{proof}
    Let $z=u-v$ be the (strong) solution of the following  problem
    \begin{align}\label{abs-z}
        z_{tt} + Az + D(u_t)-D(v_t )= F(u)-F(v),
    \end{align}
    with the  initial data $\mbf{z}_0=\mbf{u}_0-\mbf{v}_0$.
    In addition, $\mbf{z}=\{z,z_t\}$ satisfies the energy identity
    \begin{equation}\label{lem:EE-01}
        \tilde{E}(\mbf{z}(t)) + \int_s^t(D(u_t)-D(v_t),z_t)_\Omega
        = \tilde{E}(\mbf{z}(s)) + \int_s^t(F(u)-F(v),z_t)_\Omega,\quad s\leq t.
    \end{equation}
    Let $T>0$ arbitrary. Applying \eqref{lem:EE-01} for $\{s,t\}=\{t,T\}$ and integrating over $[0,T]$, we obtain
    \begin{align}\label{lem:EE-02}
        T\tilde{E}(\mbf{z}(T)) + \int_0^T\!\!\int_t^T(D(u_t)-D(v_t),z_t)_\Omega
        = \int_0^T\tilde{E}(\mbf{z})
        + \int_0^T\!\!\int_t^T(F(u)-F(v),z_t)_\Omega .
    \end{align}
    Multiplying \eqref{abs-z} by $z$ and integrating by parts over $(0,T)\times\Omega$, we obtain
    \begin{align}\label{lem:EE-04}
        2\int_0^T \tilde{E}(\mbf{z})
        = - \left.(z_t,z)_\Omega\right|_0^T
        + \int_0^T\left[ 2\|z_t\|_0^2
        - (D(u_t)-D(v_t),z)_\Omega \right]
        + \int_0^T(F(u)-F(v),z)_\Omega.
    \end{align}
    Combining \eqref{lem:EE-02} and \eqref{lem:EE-04} we have
    \begin{align}
        T\tilde{E}(\mbf{z}(T)) + \int_0^T \tilde{E}(\mbf{z})
        \leq
        - \left.(z_t,z)_\Omega\right|_0^T
        + \int_0^T\left[ 2\|z_t\|_0^2
        + |(D(u_t)-D(v_t),z)_\Omega| \right] \notag\\
        + \int_0^T(F(u)-F(v),z)_\Omega
        + \int_0^T\!\!\int_t^T(F(u)-F(v),z_t)_\Omega. \label{lem:EE-03}
    \end{align}
    
    \noindent\textit{Pointwise energy terms:}
    The pointwise energy terms in the RHS of \eqref{lem:EE-03} are handle as follows
    \begin{equation}\label{lem:EE-PWT}
        |\left.-(z,z_t)_\Omega\right|_0^T|
        \leq 2\tilde{E}(\mbf{z}(0)) + \int_0^T\!\!(Fu-Fv,z_t)_\Omega,
    \end{equation}
    where we have used the ``energy identity'' \eqref{lem:EE-01} for $\{s,t\}=\{0,T\}$.
    
    \noindent\textit{Damping terms:}
    The damping terms in the RHS of \eqref{lem:EE-03} are handled next. Since $b_0 >0 $, the assumption (\ref{aspg} )  
    \begin{align}
        \quad  b_0\|z_t\|_0^2
        \leq(D(u_t)-D(v_t),z_t)_\Omega \quad \text{and} \label{lem:EE-05}
    \end{align}
    \hspace{1pt}
    \begin{align}
        &|(D(u_t)-D(v_t),z)_\Omega|
        =
        |(g(\|u_t\|_0)u_t - g(\|v_t\|_0)v_t,z)_\Omega| \notag\\
        &\leq
        b_0|(z_t, z)|
        +\sum_{j=1}^{q}b_j|(\|u_t\|_0^ju_t-\|v_t\|_0^jv_t,z)_\Omega| \notag\\
        &\leq
        \|z_t\|_0^2 + \frac{b_0^2}{4}\|z\|_0^2
        +\sum_{j=1}^{q}b_j
        \left[
            \|u_t\|_0^j|(z_t,z)_\Omega|
            +|(\|u_t\|_0^j-\|v_t\|_0^j)||(v_t,z)_\Omega|
        \right] \notag\\
        &\leq
        \|z_t\|_0^2 + \frac{b_0^2}{4}\|z\|_0^2
        +\sum_{j=1}^{q}b_j
        \left[
            \|u_t\|_0^j\|z_t\|_0\|z\|_0
            +\|z_t\|_0\left(\sum_{l=0}^{j-1}\|u_t\|_0^{j-1-l}\|v_t\|_0^{l}\right)\|v_t\|_0\|z\|_0
        \right] \notag\\
        &\leq 
        2\|z_t\|_0^2 + \frac{b_0^2}{4}\|z\|_0^2
        + \frac{1}{2}\left(\sum_{j=1}^{q}b_j\|u_t\|_0^j\right)^2\|z\|_0^2
        + \frac{1}{2}\left[\sum_{j=1}^{q}b_j\sum_{l=0}^{j-1}\|u_t\|_0^{j-1-l}\|v_t\|_0^{l+1}\right]^2\|z\|_0^2 \notag\\
        &\leq 
        2\|z_t\|_0^2 
        +\frac{1}{2}\left\{\frac{b_0^2}{2}+\left(\sum_{j=1}^{q}b_j\|u_t\|_0^j\right)^2
        + \left[\sum_{j=1}^{q}b_j\sum_{l=0}^{j-1}\|u_t\|_0^{j-1-l}\|v_t\|_0^{l+1}\right]^2
        \right\}\|z\|_0^2. \label{lem:EE-06}
    \end{align}
    Since $B$ is bounded forward invariant, we have
    \begin{equation}\label{lem:EE-07}
        \left\{\left(\sum_{j=1}^{q}b_j\|u_t\|_0^j\right)^2
        + \left[\sum_{j=1}^{q}b_j\sum_{l=0}^{j-1}\|u_t\|_0^{j-1-l}\|v_t\|_0^{l+1}\right]^2
        \right\}
        \leq C_{B,q},\quad\mbox{for}~t\geq0,
    \end{equation}
    for some positive constant $C_{B,q}$.
    Using \eqref{lem:EE-05} and \eqref{lem:EE-07}, it follows from \eqref{lem:EE-06} that
    \begin{align}\label{lem:EE-DT}
        \int_0^T\left[2\|z_t\|_0^2+|(D(u_t)-D(v_t),z)_\Omega|\right]
        &\leq 4\int_0^T\!\!(D(u_t)-D(v_t),z_t)_\Omega
        + C_{B,q}\int_{Q^T}|z|^2.
    \end{align}
    Combining \eqref{lem:EE-03} with \eqref{lem:EE-PWT} and \eqref{lem:EE-DT}, and using the ``energy identity'' \eqref{lem:EE-01}  in order to estimate $\int_0^T ( D(u_t) - D( v_t), z_t) dt $ we obtain
    \begin{align}
        T\tilde{E}(\mbf{z}(T)) + \int_0^T \tilde{E}(\mbf{z})
        \leq
        6\left[E(\mbf{z}(0))
        + \int_0^T(F(u)-F(v),z)_\Omega
        + \int_0^T(F(u)-F(v),z_t)_\Omega \right. \notag\\ \left.
        + \int_0^T\!\!\int_t^T(F(u)-F(v),z_t)_\Omega \right]
        + C_{B,\varepsilon,q}\int_{Q^T}|z|^2. \label{lem:EE-08}
    \end{align}
    Inequality \eqref{lem:EE-ineq} thus follows from \eqref{lem:EE-08} by  adjusting the constants $C$ and $C_{R,T}$.
\end{proof}

\begin{remark}
    It is important to notice that the positive constant $C$ in Lemma \ref{lem:EE} above does not depend on time.
\end{remark}

The above Lemma provides an estimate for the energy of the difference of trajectories that are uniformly bounded in time. The above estimate  can  be  extended by density to all  generalized solution. This is done by  using density of $D(A)\times H^2_\ast$ in $\cH$.
We are now in position to prove Proposition \ref{pro:QS}.\\

{\bf Step 2: }[Completion of the Proof of Proposition \ref{pro:QS}.]
    Let $B\subset\cH$ be a bounded forward invariant set. As observed before, it is sufficient to show that $\{\cH,S_t\}$ is asymptotically quasi-stable on $B$. If $\mbf{u}_0,\mbf{v}_0\in B$ and $\mbf{u}=\{u,u_t\}$ and $\mbf{v}=\{v,v_t\}$ are the corresponding (generalized) solutions of \eqref{sbe}-\eqref{IC}, we set $\mbf{z}=\mbf{u}-\mbf{v}$, solution to the problem \eqref{abs-z}, that verifies the following ``energy identity''
    \begin{align} \label{pro:QS-01}
        \tilde{E}(\mbf{z}(t)) + \int_0^t(D(u_t)-D(v_t),z_t)_\Omega
        = \tilde{E}(\mbf{z}(0)) + \int_0^t(F(u)-F(v),z_t)_\Omega,\quad t>0.
    \end{align}
    Since the damping operator $D$ is monotone and the operator $F$ is locally Lipschitz from $H^2_\ast$ to $L^2(\Omega)$ (see Lemmas \ref{lem:op_D} and \ref{lem:op_F}, respectively), inequality \eqref{pro:QS-01} implies
    \begin{align*}
        \tilde{E}(\mbf{z}(t))
        \leq \tilde{E}(\mbf{z}(0))
        + L_B\int_0^t \tilde{E}(\mbf{z}), \quad\mbox{for}~t>0,
    \end{align*}
    where $L_B>0$ is the Lipschitz constant of $F$ on $B$.
    Applying Gronwall's inequality and setting $h(t)=e^{tL_R}$, we conclude
    \begin{equation}\label{def:AQS-01}
        \|S_t\mbf{u}_0-S_t\mbf{v}_0\|_\cH^2
        \leq h(t)\|\mbf{u}_0-\mbf{v}_0\|_\cH^2,\quad\mbox{for}~t>0.
    \end{equation}
    Note that $h(t)$ is locally bounded in $[0,\infty)$. Next, by applying Lemma \ref{lem:EE}, we have
    \begin{align}
        T\tilde{E}(\mbf{z}(T))
        + \int_0^T\!\!\tilde{E}(\mbf{z})
        \leq C\left[\tilde{E}(\mbf{z}(0))
        + \int_0^T\!\!(F(u)-F(v),z)_\Omega
        + \int_0^T\!\!(F(u)-F(v),z_t)_\Omega \right. \notag\\ \left.
        + \int_0^T\!\!\int_t^T(F(u)-F(v),z_t)_\Omega \right]
        + C_{B,T}\lot{[0,T]}{z} \label{pro:QS-02}
        ,
    \end{align}
    for every $T\!>\!0$ and for some constant $C_{B,T}\!>\!0$, where lower order terms are given in \eqref{lem:EE-lot}.
    
    \noindent\textit{Source Terms.}
   For the source $F(u) - F(v)  $   we have 
    \begin{align} \label{pro:QS-02a}
        \abs{\int_0^T(F(u)-F(v),z)_\Omega}
        \leq L_B\int_0^T\!\!\|z\|_{2,\ast}\|z\|_0
        \leq \gamma\int_0^T\!\!\tilde{E}(\mbf{z})
        + \frac{L_B^2T}{4\gamma}\lot{[0,T]}{z},
    \end{align}
    for every $\gamma>0$ (arbitrarily small) , where $L_B>0$ is the Lipschitz constant of $F$ on $B$.
    And,
    \begin{align}
        (Fu-Fv,z_t)_\Omega
       \! =\! -\!\kappa(u^+\!-\!v^+,z_t)_\Omega
        \!+\! (\Lambda{u} \!-\! \Lambda{v},z_t)_\Omega
        - (f(u)-f(v),z_t)_\Omega
        - \beta(z_y,z_t)_\Omega, \label{pro:AC-08}
    \end{align}
    where $\Lambda u=(\alpha-\delta\|u_x\|_0^2)u_{xx}$ and similar expression holds for $\Lambda{v}$. Note
    \begin{align}
        \abs{-\kappa(u^+-v^+,z_t)_\Omega}
        \leq \gamma \tilde{E}(\mbf{z}) + \frac{\kappa^2}{2\gamma}\|z\|_0^2, \label{pro:AC-08_1}
    \end{align}
    for any $\gamma>0$. Also,
    \begin{align}
        \abs{\beta(z_y,z_t)_\Omega}
        \leq \gamma \tilde{E}(\mbf{z})
        + \frac{\beta^2}{2\gamma}\|z\|_{2-\eta}^2, \label{pro:AC-08_2}
    \end{align}
    for any $\gamma>0$ and some $\eta\in(0,2]$.
    Finally, calculations involving the $\Lambda$ operator  [terms of critical regularity] will follow the steps in \cite{BGLW22}. Integrating by parts in $\Omega$, and having in mind that the solutions $\mbf{u}=\{u,u_t\}$ and $\mbf{v}=\{v,v_t\}$ are strong, we have
    \begin{align}
        (\Lambda{u}-&\Lambda{v},z_t)_\Omega
        = \alpha(z_{xx},z_t)_\Omega
        + \delta(\|u_x\|_0^2u_{xx} - \|v_x\|_0^2v_{xx},z_t)_\Omega \notag\\
        &= \!- \dfrac{d}{dt}\left\{\frac{1}{2}\left[\alpha \!
        +\!\delta\|u_x\|_0^2\right]\|z_x\|_0^2\right\}
        \!-\! \delta\|z_x\|_0^2(u_{xx},u_t)_\Omega 
        + \delta[\|u_x\|_0^2\!-\!\|v_x\|_0^2](v_{xx},z_t)_\Omega  \label{pro:AC-08_3}
    \end{align}
    Integrating \eqref{pro:AC-08} over $[s,t]$, for $0\leq s < t\leq T$, and using \eqref{pro:AC-08_1}, \eqref{pro:AC-08_2} and \eqref{pro:AC-08_3} we obtain
    \begin{align*}
        \int_s^t\!\!(Fu-Fv,z_t)_\Omega
        &\leq 
        2\gamma\int_{\textcolor{black}{s/0}}^{\textcolor{black}{t/T}}\!\!\tilde{E}(\mbf{z})
        + \int_s^t\!\!(\Lambda{u}-\Lambda{v},z_t)_\Omega \notag\\
        &\quad
        \textcolor{black}{- or +} \int_s^t\!\!(fu-fv,z_t)_\Omega
        + \left(\frac{\kappa^2+\beta^2}{2\gamma}\right)(t-s)\sup_{[0,T]}\|z\|_{2-\eta}^2 \notag\\
        &\leq
        2\gamma\int_{0}^T\!\!\tilde{E}(\mbf{z})
        + \abs{\left.\frac{1}{2}[\alpha-\delta\|u_x\|_0^2]\|z_x\|_0^2\right|_s^t}
        + \delta\int_s^t\|z_x\|_0^2\|u_{xx}\|_0\|u_t\|_0 \notag\\
        &\quad
        + \delta\int_s^t\!\!\|z_x\|_0[\|u_x\|_0+\|v_x\|_0]\|v_{xx}\|_0\|z_t\|_0 \notag\\
        &\quad
        + \int_s^t\!\!(f(u)-f(v),z_t)_\Omega
        + \left(\frac{\kappa^2+\beta^2}{2\gamma}\right)(t-s)\sup_{[0,T]}\|z\|_{2-\eta}^2.
    \end{align*}
    Using once again the fact that solutions $\mbf{u}=\{u,u_t\}$ and $\mbf{v}=\{v,v_t\}$ belong to $B$ and, consequently, are uniformly bounded in $\cH$, we conclude
    \begin{align}
        \int_s^t\!\!(F(u)-F(v),z_t)_\Omega
        &\leq 3\gamma\int_s^t\!\!\tilde{E}(\mbf{z})
        + \int_s^t\!\!(f(u)-f(v),z_t)_\Omega
        + C_{B,T,\gamma}\sup_{[0,T]}\|z\|_{2-\eta}^2 \label{pro:QS-02b}
    \end{align}
    for $s\leq t$, every $\gamma>0$ and for some positive constant $C_{T,B,\gamma}$.
    Returning to \eqref{pro:QS-02} and using \eqref{pro:QS-02a} and \eqref{pro:QS-02b}, we obtain
    \begin{align}
        T\tilde{E}(\mbf{z}(T))
        + \int_0^T\!\!\tilde{E}(\mbf{z})
        \leq C\left[\tilde{E}(\mbf{z}(0))
        + \gamma(4+3T)\int_0^T\!\!\tilde{E}(\mbf{z}) 
        + (1+T)\int_0^T\abs{(f(u)-f(v),z_t)_\Omega} \right]
        \notag\\
        + \left[C_{B,T} + C\left(\frac{L_B^2T}{4\gamma} + (1+T)C_{B,T,\gamma}\right)\right]\lot{[0,T]}{z}. \label{pro:QS-03}
    \end{align}
    The last source term in the RHS of \eqref{pro:QS-03} is handled using the $C^1(\R)$ property of $f$ and the embedding $H^2_\ast\subset L^\infty(\Omega)$, as follows.
               \begin{align}
                          (1+T)\int_0^T&\abs{(f(u)-f(v),z_t)_\Omega}
                         \leq (1+T)\int_0^T\max_{|s|\leq\|u\|_{L^\infty(\Omega)}+\|v\|_{L^\infty(\Omega)}}|f'(s)| \|z\|_0\|z_t\|_0 \notag\\
                        &\leq \underbrace{(1+T)\max_{|s|\leq R_B}|f'(s)|}_{C_{f,T,B}}\int_0^T\|z\|_0\|z_t\|_0 \leq \gamma\int_0^T\!\!\tilde{E}(\mbf{z})
                          + \frac{T[C_{f,T,B}]^2}{4\gamma}\lot{[0,T]}{z}. \label{pro:QS-03a}
                \end{align}
    Combining \eqref{pro:QS-03} and \eqref{pro:QS-03a}, we conclude
                  \begin{align*}
                           T\tilde{E}(\mbf{z}(T)) + &(1-\gamma(5+3T)C)\int_0^T\!\!\tilde{E}(\mbf{z}) \notag\\
                          &\leq C\tilde{E}(\mbf{z}(0))
                           + \left[C_{B,T} + C\left(\frac{L_B^2T}{4\gamma} + (1+T)C_{B,T,\gamma} \frac{T[C_{f,T,B}]^2}{4\gamma}\right)\right]\lot{[0,T]}{z}.
                 \end{align*}
    Finally, choosing $\gamma=\gamma(T)>0$ small enough and $T>0$ sufficiently large, we arrive at
    \begin{align}\label{pro:QS-04}
        \tilde{E}(\mbf{z}(T))
        \leq &\underbrace{\frac{C}{T}}_{\theta}\tilde{E}(\mbf{z}(0))
        + \underbrace{\frac{1}{T}\left[C_{B,T} + C\left(\frac{L_B^2T}{4\gamma} + (1+T)C_{B,T,\gamma} + \frac{T[C_{f,T,B}]^2}{4\gamma}\right)\right]}_{C_{0,T,B}}\lot{[0,T]}{z} \notag\\
       &\leq \theta \tilde{E}(\mbf{z}(0)) + C_{0,T,B}\lot{[0,T]}{z},
    \end{align}
    where $0<\theta<1$ and $C_{0,T,B}>0$ is constant. Let us consider $m\in\N$. By standard iteration via semigroup property, the above inequality \eqref{pro:QS-04} implies
    \begin{align*}
        \tilde{E}(\mbf{z}(mT)) 
        \leq \theta^m \tilde{E}(\mbf{z}(0)) + C_{0,T,B}\sum_{j=0}^{m-1}\lot{[(m-j)T,(m-j-1)T]}{z}.
    \end{align*}
    Therefore, if $t>0$ then there exists $m\in\N$ such that $t=mT+r$, where $r\in[0,1]$ and, consequently,
    \begin{align*}
        \tilde{E}(\mbf{z}(t)) = C_A\tilde{E}(\mbf{z}(mT))
        \leq C_Ae^{-r\ln\theta/T}e^{\frac{\ln\theta}{T}t}\tilde{E}(\mbf{z}(0)) + C_AC_{0,T,B}\lot{[0,t]}{z},
    \end{align*}
    where $C_A>0$ is constant that depend on the linear semigroup generated by the plate operator $A$, and do not depend on $t$. Finally, the above inequality can be rewritten as follows
    \begin{equation}\label{def:AQS-02}
        \|S_t\mbf{u}_0-S_t\mbf{v}_0\|_\cH^2
        \leq \underbrace{C_Ae^{\frac{|\ln\theta|}{T}}}_{C_{A,T}}e^{\frac{\ln\theta}{T}t}\|\mbf{u}_0-\mbf{v}_0\|_\cH^2
        + \underbrace{2C_AC_{0,T,B}}_{d(t)} \lot{[0,t]}{z}.
    \end{equation}
    Set $g_0(t) = C_{A,T}e^{\frac{\ln\theta}{T}t}$, $d(t)=2C_AC_{0,T,B}$ and $\mu_{H^2_\ast}(u(s)-v(s))=\|u(s)-v(s)\|_{2-\eta}^2$, for $s\in[0,t]$ and for any $\eta\in(0,2]$. Notice that $l(t)$ is constant and, therefore, locally Lipschitz in $[0,\infty)$. Also $g_0(t)\in L^1(\R)$ with property $g_0(t)\to0$ as $t\to\infty$ ($0<\theta<1$). Finally, $\mu_{H^2_\ast}$ is a compact seminorm in $H^2_\ast$. 
    
    Inequalities \eqref{def:AQS-01} and \eqref{def:AQS-02} implies that $\{\cH,S_t\}$ is asymptotically quasi-stable on $B$ and, consequently, quasi-stable on $B$. Since $B$ is an arbitrary bounded forward invariant set on $\cH$, the proof is complete.
\end{proof}

\subsubsection[Dimension and regularity]{Fractal Dimension and Regularity of the Attractor -Completion of the Proof od Part II, Theorem \ref{thm:Att} }
The following geometrical properties of the attractor are a direct consequence of the asymptotic quasi-stability property proved in the previous section.

\begin{corollary}[Finite Fractal Dimension]\label{cor:FFD}
    $\mathfrak{A}$ has finite fractal dimension in $\cH$.    
\end{corollary}
\begin{proof}
    Under the assumption that $b_0 >0 $, the finite dimension of $\mathfrak{A}$ follows from Theorem 3.4.5 on \cite{Chu15} page 131.
\end{proof}

Since $\{\cH,S_t\}$ is quasi-stable on $\mathfrak{A}$ (see Definition \ref{def:QS}) for every $T>0$, an estimate for the fractal dimension of the attractor is also provided by Theorem 3.4.5 in \cite{Chu15}, namely
\begin{equation*}
    \textrm{dim}_f\mathfrak{A} \leq \left[\ln\frac{2}{1+p}\right]^{-1}\cdot\ln \textrm{m}_Z\left(\frac{4L_K}{1-p}\right)
\end{equation*}
where $0<p<1$ (depending on $T$),
$$Z=W_1(0,T)\equiv\left\{z\in L^2(0,T;H^2_\ast):\|z\|_Z^2\equiv \|z\|_{L^2(0,T;H^2_\ast)}^2+\|z_t\|_{L^2(0,T;L^2(\Omega))}^2<\infty\right\}$$
and $L_K$ is the Lipschitz constant of $K:\cH\to Z$ given by $K\mbf{y}_0 = y$ where $\{y(t),y_t(t)\}=S_t\mbf{y}_0$ is the corresponding (generalized) solution of \eqref{sbe}-\eqref{IC}, and $\textrm{m}_Z(R)$ is the maximum number of elements $z_i$ in the ball $\{z\in Z:\|z\|_Z\leq R\}$ possessing the property $$1<\textrm{n}_Z(z_i-z_j)\equiv C_T\sup_{s\in[0,T]}\|z_i(s)-z_j\|_0$$ 
for $i\neq j$, where $C_T>0$ is constant.

\begin{corollary}[Regularity]\label{cor:Reg}
 Under the assumption that $b_0>0$, any full trajectory \\ \(\gamma=\left\{\mbf{u}(t)=(u(t),u_t(t)):t\in\R\right\}\) that belongs to $\mathfrak{A}$ enjoys the  following regularity properties:
    \begin{equation*}
        u_t\in L^\infty(\R,H^2_\ast)\cap C(\R,L^2(\Omega))\quad\mbox{and}\quad u_{tt}\in L^\infty(\R,L^2(\Omega)).
    \end{equation*}
    Moreover, there exists $R>0$ such that
    \begin{equation*}
        \|u_t(t)\|_{2,\ast}^2 + \|u_{tt}(t)\|_0^2 \leq R,\quad t\in\R.
    \end{equation*}
\end{corollary}
\begin{proof}
    As in the proof of Proposition \ref{pro:QS}, the dynamical system $\{\cH,S_t\}$ is asymptotically quasi-stable on $\mathfrak{A}$, with \textit{stabilizability inequality}
    \begin{align}
        \|S_t\mbf{u}_0-S_t\mbf{v}_0\|_\cH^2
        \leq g_0(t)\|\mbf{u}_0-\mbf{v}_0\|_\cH^2
        + l(t)\sup_{0\leq s\leq t}\left[\mu_{H^2_\ast}(u(s)-v(s))\right]^2,\notag
    \end{align}
    for $\mbf{u}_0,\mbf{v}_0\in\mathfrak{A}$, 
    where $u$ and $v$ are the corresponding (generalized) solutions, and 
    $g_0(t) = ce^{-\omega t}$, for some constants $c>0$ and $\omega>0$, $l(t)>0$ is constant for every $t$, and $\mu_{H^2_\ast}(z)=\|z\|_{2-\eta}$, for $\eta\in(0,2]$ is a compact seminorm in $H^2_\ast$. Therefore, the conclusion follows from Theorem 3.4.19 in \cite{Chu15} page 140.
\end{proof}

\section{Proof of  Abstract Theorem \ref{0.5.1}}\label{ultdisabs}
\textbf{ Ultimate dissipativity for a model with non-conservative forces:}
In this section we provide a proof of the ultimate dissipativity property for a general abstract model with non-conservative force $ N(u)$.  Thus, the resulting model is non-dissipative where  classical Lyapunov's function approaches fail, as recognized in \cite{Memoir,BGLW22,pata} where various scenarios of  non-dissipative models were studied.  Particularly challenging is the case when {\it nonlinear damping }is also present in the system. This is the case under consideration below. The main idea of the proof is based nn the so called ``barrier's method'' inspired by works of \cite{haraux} where questions of ultimate boundedness of nonlinearly damped  abstract ``harmonic oscillator'' with non-autonomous forcing has been considered. 

With reference to the system $(\cH,S_t)$, under the assumptions formulated in Theorem  \ref{0.5.1},  the positive energy, the total energy  and the energy relation associated  with \eqref{model2} have the following form:
            \begin{align}
                  &E(u(t),u_t(t))\!=\!\frac{1}{2}\left(|\cA^{1/2}(u)|^2+ | u_t|^2\right)\!+\!\Pi_0(u), \quad \cE(u(t),u_t(t))\!=\!E(u(t),u_t(t))\!+\!\Pi_1(u),
                  \label{en1} \\ 
                 &\cE(u(t),u_t(t))+k\int_0^t(Du_t(s),u_t(s))ds= \cE(u_0,u_1)+\int_0^t(N(u(s)),u_t(s))ds.\label{en2}
            \end{align}
We also have the following inequality on energy which follows from the assumption \eqref{1.5} for $\eta=\frac{1}{4}$ and $c=c_\eta=c_{1/4}$;
     \begin{equation}\label{EnergyIn}
         \frac{1}{2}E(u,u_t)-c \leq \cE(u,u_t) \leq 2E(u,u_t)+c.
    \end{equation}
    
\begin{proof}{Proof of Theorem \ref{0.5.1}:}
   The proof is based on the barrier's method as discussed before. \\
   
   {\bf Step 1:} [Preliminaries.]
To begin with, consider the following   family of functionals
     \begin{equation}\label{lya}
         V_\varepsilon(u,u_t)= \cE(u,u_t)+\varepsilon (u_t,u), \quad \text{$\varepsilon >0$ .
         }
     \end{equation}
     Since $\cD(\cA^{\frac{1}{2}})$ is continuously embedded in $L_2(\Omega) $  we have that 
     \begin{equation}\label{eniq}
        (\cI u_t,u) \leq \frac{1}{2}\left(|u_t|^2+\lambda |\cA^{\frac{1}{2}}u|^2\right)\leq max\{1,\lambda\}E(u,u_t),
     \end{equation} where $\lambda> 0$ is the (embedding) constant. 
     Combining above inequality with \eqref{EnergyIn}, we deduce for some constants $C_1, C_2, c$,
     \begin{align}\label{V-E}
         C_1\cdot E(u,u_t)-c \leq V_\varepsilon(u, u_t) \leq C_2 \cdot E(u,u_t)+c.
     \end{align}
  Differentiating the function \eqref{lya} [justified for strong solutions], and replacing $\cE$ using $\eqref{eniq}$ yields, 
           \begin{align}
              \frac{d}{dt}V_\varepsilon&(u,u_t)
         = -\kappa(Du_t, u_t)+(N(u),u_t)+\varepsilon \cdot \left((u_t,u_t)+ (u_{tt},u) \right)\notag\\
         &=-\kappa(Du_t, u_t)\!+\!(N(u),u_t)+\varepsilon [c_0+c_1(Du_t, u_t)]+ \varepsilon \! \cdot \! \left(-\cA u-\kappa Du_t-\Pi^\prime(u)\!+\!N(u),u\right)\notag\\
               \intertext{\text{ using \eqref{A3b} and then assumption \eqref{A2}, \eqref{A3a} on the non-conservative term, we have}}
          & \quad \quad \quad +\varepsilon \left( \eta |\cA^{1/2} u|^2-c_2 \Pi_0(u)+c_3+c_4\left[1+E(u,u_t)\right]^\gamma (Du_t, u_t)\right) +\varepsilon c_0 \notag \\ 
          & \leq (-\kappa+\eta \kappa+\varepsilon c_1)(Du_t,u_t)+\delta E(u,u_t)-\varepsilon((1-\eta) |\cA^{1/2} u|^2-c_2 \Pi_0(u)) \notag \\
          & \quad \quad \quad +b(1/\delta)+\varepsilon c_4\left[1+E(u,u_t)\right]^\gamma (Du_t, u_t)+\varepsilon (c_0+c_3)\notag \\ 
            \intertext{\mbox{since $\!\varepsilon \!\left(\!-\!(1\!-\!\eta)|\cA^{1/2}u|^2\!-\!c_2\Pi_0(u) \!\right)\!\leq\! -\! 2d^\prime \varepsilon E(u,u_t)\!+\!d^\prime \varepsilon(u_t,u_t), d^\prime\!=\!min\{\!(1\!-\eta),c_2\}$}}
               & \quad  \text{using \eqref{A2}, }(u_t,u_t)\! \leq \! c_0+c_1(Du_t,u_t)] \text{and therefore}\notag \\
          & \leq (-\kappa+\eta \kappa+2\varepsilon c_1) (Du_t,u_t)+(\delta-2d^\prime \varepsilon) E(u,u_t) \notag \\
          & \quad \quad \quad +b(1/\delta)+\varepsilon \left(c_4\left[1+E(u,u_t)\right]^\gamma \right)(Du_t, u_t)+\varepsilon (2c_0+c_3) \notag \\
              \text{ since  }\eta<1,& \mbox{ suitable choice of } \varepsilon \mbox{ gives us } d_3'>0, \mbox{ where } d_3'=\kappa(1-\eta)-2\varepsilon c_1 \notag \\
            \intertext{\mbox{ as well as choose appropriate  $\delta $ proportional to $\varepsilon$  so that  $\delta< d^{\prime}\varepsilon$, for $d_0'=2c_0+c_3$ }}
          & \leq -\varepsilon E(u,u_t)+\varepsilon d_0'+b(1/\delta)+\big(-d_3'+\varepsilon \cdot c_4 \left[1+E(u,u_t)\right]^\gamma \big)(Du_t, u_t). \notag
    \end{align}
   Using the relation obtained in \eqref{V-E} in between the positive energy and the Lyapunov functional we get ($V_\varepsilon(t)$ should be understood as $V_\varepsilon(u(t),u_t(t))$),
\begin{equation*}
    \frac{d}{dt}V_\varepsilon(t) \leq -\frac{\varepsilon}{C_2}(V_\varepsilon(t)-c)+\varepsilon d_0'+b(1/\delta)+\big(\!\varepsilon \cdot c_4\left[1+E(t)\right]^\gamma -d_3'\big)(Du_t(t), u_t(t)),
   \end{equation*}
   once again by rescaling $\varepsilon:=\varepsilon/C_2$, 
   \begin{equation*}
       \frac{d}{dt}V_\varepsilon (t)\leq -\varepsilon V_\varepsilon(t)+\varepsilon c +\varepsilon d_0'+b(1/\delta)+\big(\varepsilon \cdot c_4 \left[1+E(u,u_t)\right]^\gamma -d_3' \big)(Du_t(t), u_t(t)).
   \end{equation*}
 For $d_0=c+d_0', d_2=c_4,  d_3=d_3'/d_2$ \mbox{ and }$ d_1$,\mbox{ satisfies } $b(d_1/\delta)\geq \frac{1}{d_0}b(1/\delta),$ \hfill \break
 rewriting the above inequality in terms of the new constants,
\begin{align}
 \frac{d}{dt}&V_\varepsilon(t) +\varepsilon V_\varepsilon(t) \leq d_0\!\left\{\varepsilon+ b\left(\frac{d_1}{\varepsilon}\right)\!\right\}+d_2\left\{\varepsilon \left[ 1+E(t)\right]^{\gamma} -d_3 \right\}\cdot(Du_t(t),u_t(t)), \notag \\
                      \intertext{\text{ multiplying by $e^{\varepsilon t}$ and integrating on $t\geq s \geq 0$ we get,}}
    & V_\varepsilon(t) \leq e^{-\varepsilon(t-s)}V_\varepsilon(s)\!+ \! d_0\! \left\{\!\!1\!+\! \frac{1}{\varepsilon} b\left(\frac{d_1}{\varepsilon}\right)\!\!\right\}\notag \\
        &\quad \quad \quad \quad +\!d_2\int_s^t \!\!e^{-\varepsilon(t\!-\!\tau)}\left\{\varepsilon \left[\! 1\!+\!E(\tau)\right]^{\gamma}\! -\!d_3 \right\}\!\cdot\!(Du_t(\tau),u_t(\tau)\!) d\tau \label{55a}\\
   & \mbox{or, }C_1 \cdot E \left(t \right)-c \leq e^{-\varepsilon(t-s)}\left(C_2 E(s)+c\right)+d_0\left\{1+\frac{1}{\varepsilon} b\left(\frac{d_1}{\varepsilon}\right)\right\}\notag \\
         &\hspace{2cm}+d_2 \int_s^te^{-\varepsilon(t-\tau)}\left\{\varepsilon \left[ 1+E(\tau)\right]^{\gamma} -d_3 \right\}(Du_t(\tau),u_t(\tau) )d\tau \notag\\
               \intertext{\text{using property of exponential function and rescaling of $d_i$  after dividing by $C_1$, we obtain}}
       & E \left(t \right) \leq (C_2/C_1) \cdot E(u(s),u_t(s))+{\frac{2c}{C_1}}+d_0\left\{1+\frac{1}{\varepsilon} b\left(\frac{d_1}{\varepsilon}\right)\right\}\notag \\
       &\hspace{2cm}+d_2 \int_s^t e^{-\varepsilon\cdot(t-\tau)}\left\{\varepsilon \left[ 1+E(\tau)\right]^{\gamma} -d_3 \right\}(Du_t(\tau),u_t(\tau)) d\tau. \label{55b}
\end{align} 
 Note that achieving dissipativity from the above relation is straightforward when $\gamma=0$ by simply selecting $\varepsilon<d_3$. However, when $\gamma>0$, the integrand grows significantly. Trying to control  it by choosing a small $\varepsilon$ results in losing control over the third term, as its growth is impacted by $\frac{1}{\varepsilon}$. To address this second case, we continue the proof  in three distinct steps. Step 1 precisely shows that for a particular choice of $\varepsilon$ the integrand in the third term can be controlled.
 \paragraph{Step  2.} 
   We are now in position to apply barrier's method. The goal is to prove \eqref{55a} without the integral term. Let $\sigma:\mathbb{R}^+ \to \mathbb{R}^+\!-\!\{0\}$ is a map defined by $ E(t) \mapsto \sigma(E(t)) $  where $\sigma(E(t))$ is the solution of the following equation [we denote $\sigma(E(t))$ by $\sigma(t)$]

 \begin{equation}\label{sigma}
     \left[1+\frac{C_2}{C_1}E(t)+\frac{2c}{C_1}+d_0\left\{1+\sigma(t) \cdot b(d_1\sigma(t))\right\}\right]^\gamma=\frac{d_3\sigma(t)}{2}.
 \end{equation}
It is clear that the solution of the above equation not only exists but also unique  due to (\ref{balance}) and  as the left side is a continuous increasing function which exhibits slower growth rate than the right side due to $\gamma <1$. 
   We claim that $\mbox{for } \varepsilon_s:=\left[\sigma(s)\right]^{-1} $(the inverse function) the integrand 
 \begin{equation}\label{claim}
     \varepsilon_s \left[ 1+E(t)\right]^{\gamma} -d_3 \leq 0 \quad \mbox{ for all } t\geq s. 
 \end{equation} 
\begin{align*}\label{t=s}
      \mbox{at }t=s,\quad\varepsilon_s \left[ 1+E(s)\right]^{\gamma} -d_3 &\leq \varepsilon_s \left[1+\frac{C_2}{C_1}E(s)+\frac{2c}{C_1}+d_0\left\{1+\sigma(s) \cdot b(d_1\sigma(s))\right\}\right]^\gamma-d_3\notag \\
      &\leq \varepsilon_s \cdot \frac{d_3 \cdot \sigma(s)}{2}-d_3\leq\frac{-d_3}{2}<0,
  \end{align*}
  and the function $\varepsilon_s \left[ 1+E(t)\right]^{\gamma} -d_3$ is continuous hence there is an interval $[s,s+T)$ where \eqref{claim} holds. If $T$ is infinite then we are done. In the case of $T \!<\! \infty$ there must exists $T^\ast\!>\!s$ so that         \begin{equation}\label{Tast}
               \begin{cases}
                      \varepsilon_s \left[ 1+E(t)\right]^{\gamma} -d_3 < 0;\text{  for } t \in [s, s+T^\ast), \\
                     \varepsilon_s \left[ 1+E(t)\right]^{\gamma} -d_3 = 0;\text{  for } t=T^\ast
               \end{cases}
             \end{equation}
             \begin{equation}\label{Tast2}
                \mbox{applying \eqref{Tast} on \eqref{55a} we get, }\quad E \left(t \right) \leq \frac{C_2}{C_1} \cdot E(s)+{\frac{2c}{C_1}}+d_0\left\{1+\frac{1}{\varepsilon_s} b\left(\frac{d_1}{\varepsilon_s}\right)\right\}.
            \end{equation}
However, by \eqref{sigma} and \eqref{Tast2} give us, for $t \in [s, s+T^\ast)$
                   \begin{align}
                              \varepsilon_s \! \left[ 1\!+\!E(t)\right]^{\gamma} \!-\!d_3 \!\leq \!\varepsilon_s \left[ 1\!+\!\frac{C_2}{C_1} \cdot E(s)\!+\!{\frac{2c}{C_1}}\!+\!d_0\left\{1\!+\!\frac{1}{\varepsilon_s} b\left(\frac{d_1}{\varepsilon_s}\right)\right\}\right]^{\gamma} \!\!-\!d_3 \!\leq\! \varepsilon_s \frac{d_3\sigma(s)}{2}\!-\!d_3\!=\!\frac{-d_3}{2}.\notag
                      \end{align}

This is a contradiction to the fact that the function $\varepsilon \left[ 1+E(t)\right]^{\gamma} -d_3 $ is continuous at $T^\ast$.
 Which concludes that $\varepsilon \left[ 1+E(u(t),u_t(t))\right]^{\gamma} -d_3 \leq 0; \hspace{.4cm} \forall t\geq s$ when $\varepsilon=\sigma\left[E(u(s),u_t(s))\right]^{-1}.$
 \\
According to \eqref{claim}, we have
           \begin{equation}\label{21}
                      V_\varepsilon(t) \leq e^{-\varepsilon(t-s)}V_\varepsilon(s)+  d_0\left\{1+\frac{1}{\varepsilon} b\left(\frac{d_1}{\varepsilon}\right)\right\}\text{for }  t\geq s \geq 0,
           \end{equation}
  where,  $\varepsilon=\sigma(E(s))^{-1}$ and $\sigma$ satisfies \eqref{sigma}. In particular for $s=0, \varepsilon=\sigma(E(0))^{-1}$
              \begin{align}\label{22}
                      &V_\varepsilon(t) \leq e^{-\varepsilon t}V_\varepsilon(0)+  d_0\left\{1+\frac{1}{\varepsilon} b\left(\frac{d_1}{\varepsilon}\right)\right\}\text{for }  t \geq 0 \notag \\
                       \text{hence, }&V_\varepsilon(t) \leq V_\varepsilon(0)+  d_0\left\{1+\frac{1}{\varepsilon} b\left(\frac{d_1}{\varepsilon}\right)\right\}\text{for }  t \geq 0.
                \end{align}
\paragraph{Step 3.}  In this step we will prove that $V_\varepsilon(t)$ is uniformly bounded for all $t>0$ when $V_\varepsilon(u_0,u_1)\leq R$ for some $R>0$. From \eqref{22} we have,
        \begin{align}\label{23}
                &V_\varepsilon(t) \leq R+  d_0\left\{1+\frac{1}{\varepsilon_R} b\left(\frac{d_1}{\varepsilon_R}\right)\right\}=K_R(say)  \text{ for }t \geq 0 \quad \mbox{ implies }\sigma(V_\varepsilon(t))<\sigma(K_R) \notag\\
             &\mbox{or,  }\varepsilon(V_\varepsilon(s))>\varepsilon(K_R)= \varepsilon_R(say) \text{ for }s \geq 0,  \notag\\
                   &\text{which implies, } V_\varepsilon(t)\leq e^{-\varepsilon(V_\varepsilon(s)) (t-s)}V_\varepsilon(s)+d_0\left\{1+\frac{1}{\varepsilon} b\left(\frac{d_1}{\varepsilon}\right)\right\}\notag\\
                &\text{hence, }V_\varepsilon(t)\leq e^{-\varepsilon_R (t-s)}V_\varepsilon(s)+d_0\left\{1+\frac{1}{\varepsilon(V_\varepsilon(s))} b\left(\frac{d_1}{\varepsilon(V_\varepsilon(s))}\right)\right\} \text{ for }  t\geq s \geq 0.
           \end{align}
Define, $W_R(t)=\sup_{V_\varepsilon(0)\leq R}V_\varepsilon(t)$. Then from \eqref{23} we have, 
         \begin{align}\label{31}
                 &W_R(t)\leq e^{-\varepsilon_R (t-s)}W_R(s)+d_0\left\{1+\frac{1}{\varepsilon(W_R(s))} b\left(\frac{d_1}{\varepsilon(W_R(s))}\right)\right\} \text{ for }  t\geq s \geq 0 \notag \\
              & \text{ when } t\to \infty, \quad  W_R(\infty)\leq d_0\left\{1+\frac{1}{\varepsilon(W_R(s))} b\left(\frac{d_1}{\varepsilon(W_R(s))}\right)\right\} \text{ for }  s \geq 0 \notag \\
              & \text{ since } \varepsilon=\sigma^{-1},\quad W_R(\infty)\leq d_0\left\{1+\sigma(W_R(s)) b\left(\sigma(W_R(s))\right)\right\} \text{ for }  s \geq 0 \notag \\
               &\implies 1 \leq \frac{d_0}{W_R(\infty)}+\frac{d_0 \cdot \sigma(s) b\left(\sigma(s)\right) }{W_R(\infty)}, \quad \text{where } \sigma(s)=\sigma(W_R(s)).
             \end{align}
Note that it is sufficient to prove that $W_R (\infty)$ is bounded uniformly [does NOT depend on $R$] in order to claim the uniform bound on $V_\epsilon$.\\
 We prove this by contradiction. If our claim is not true then $W_R(\infty)\to\infty$, and we consider the following two cases:
\begin{itemize}
    \item Case 1: When $\sigma(s) b\left(\sigma(s)\right) \leq M$ for some $M>0; \forall s>0$, then \eqref{31} implies $1 \leq 0$!! Hence $W_R(\infty)$ is uniformly bounded from above.
    \item Case 2: When $\sigma(s) b\left(\sigma(s)\right)$ is not uniformly bounded, we observe the following:
    {\allowdisplaybreaks
                  \begin{align}\label{32}
                       &[1+C\cdot E+c_0+d_0\left(1+\sigma \cdot b(d_1\sigma)\right)]^\gamma=\frac{d_3\sigma}{2} \notag \\
                            &  \implies \frac{(1+c_0+d_0)}{\sigma^{1/\gamma}}+\frac{C \cdot E}{\sigma^{1/\gamma}}+d_0\frac{\sigma \cdot b(d_1\sigma)}{\sigma^{1/\gamma}}=\left[\frac{d_3}{2}\right]^{\frac{1}{\gamma}}, \notag \\
                                 \intertext{\text{taking limit $\sigma \to \infty,  \quad \lim_{\sigma \to \infty}\frac{C \cdot E}{\sigma^{1/\gamma}}+d_0 \cdot \lim_{\sigma \to \infty}\sigma^{1-1/{\gamma}}b(d_1\sigma)=\left[\frac{d_3}{2}\right]^{\frac{1}{\gamma}}$, }}
                            \text{using \eqref{A3a} we }\text{conclude that,}   \lim_{\sigma \to \infty}&\frac{C \cdot E}{\sigma^{1/\gamma}}=\left[\frac{d_3}{2}\right]^{\frac{1}{\gamma}}.
                 \end{align}}
     Since $V_\varepsilon$ is bounded by $E$ from below by \eqref{V-E}, using the above limit in \eqref{32} we obtain, $\lim_{\sigma \to \infty}\frac{V_\varepsilon}{\sigma^{1/\gamma}}\geq d_\gamma$ for some $d_\gamma(=C_1\left[\frac{d_3}{2}\right]^{\frac{1}{\gamma}}-c)$. On the other hand we have from \eqref{31} that, when $W_R(\infty) \to \infty$, 
    \begin{equation}\label{33}
        1 \leq 0+\lim_{W_R(\infty) \to \infty}\frac{d_0 \cdot \sigma(s) b\left(\sigma(s)\right) }{W_R(\infty)}\leq \lim_{W_R(\infty) \to \infty}\frac{d_0 \cdot \left( \sigma(s)\textcolor{black}{/ \sigma(s)^{\frac{1}{\gamma}}}\right) b\left(\sigma(s)\right) }{W_R(\infty)/\textcolor{black}{\sigma(s)^{\frac{1}{\gamma}}}}\leq \frac{d_0 \cdot 0}{d_\gamma}=0.
    \end{equation}
  We used the increasing properties of the function $\sigma$ and \eqref{32}.
    This contradicts  $W_R(\infty)\to \infty$. 
    \end{itemize}
    Hence, $W_R(\infty)$ is uniformly bounded i.e., there exists $V^\ast$(independent of $R$) so that, $W_R(\infty)<V^\ast$ and dissipativity follows from \eqref{V-E}. 
    We construct the forward invariant absorbing set $\cB_0$ by $\cB_0 := \bigcup_{t \geq t_{\cB}}S_t \cB$, where $\cB=\{(u,v) \in \cH: \|(u,v)\|_{\cH}^2 \leq R\}$ is absorbing set i.e., $S_t\cB \subset \cB$. This completes the proof of  abstract Theorem \ref{0.5.1}.
 \end{proof}
 \ifdefined\xxxxxx
\section{Appendix}
In this section, we prove the inequality that we have been using throughout the paper.
\begin{lemma}
Let $u\!=\!u(x, y, t)\! \in\! \cH^1(\Omega)$, where $\Omega\!=\![0, \pi] \times [-l, l]$  such that,
$u=0$ on $\{0, \pi \}\times [-l, l]$, there exists a constant $C>0$ such that $\int_\Omega |u|^2 \leq C\cdot \int _\Omega |u_x|^2$.
\end{lemma}
\begin{proof}
By the fundamental theorem of calculus, 
\begin{align}
    &  \text{since, } u(0,y)=0, \quad\int_0^x u_x(z,y)dz=u(x,y) \quad
     \mbox{implies,} \quad |u(x,y)| \leq  \int_0^x |u_x(z,y)|dz \notag \\
    & \text{applying Holder's Inequality, one can get,}\quad |u(x,y)| \leq  \left( \int_0^x |u_x(z,y)|^2dz \right)^\frac{1}{2} \cdot \left( \int_0^x |1|^2dz \right)^\frac{1}{2} \notag \\
    & \mbox{hence,} \quad |u(x,y)|^2 \!\leq\! \pi \cdot \int_0^\pi |u_x(z,y)|^2dz.\quad \text{Integrating on both sides with respect to $y$ we get,} \notag \\
    & \int_{-l}^{l} |u(x,y)|^2 dy \leq {\pi} \cdot \int_{-l}^{l}\int_0^\pi |u_x(z,y)|^2dz dy={\pi}\cdot \int_\Omega |u_x|^2 d\Omega. \quad \text{Once again integrating both }\notag \\
    &\mbox{the sides with respect to $x$,} \int_{\Omega} |u(x,y)|^2 dy dx \leq \left( {\pi}\cdot  \int_\Omega |u_x|^2 d\Omega \right)\int_0^\pi dx = C \cdot \int_\Omega |u_x|^2 d\Omega \notag
    \end{align}
    Hence, the theorem is proved with $C=\pi^{2}$.
\end{proof}

\fi
{\bf Acknowledgments}

Research work of I. Lasiecka has been partially supported by the National Science Foundation, NSF Grant DMS-2205508.

\bibliographystyle{abbrv} 

\end{document}